\def\lastedit{2010-09-3 by jhr + corr on oct 14 by sg}
\documentclass[a4paper,12pt,draft%
]{article}
\usepackage{geometry} % geometry.sty
\usepackage{amsfonts,amssymb} % For Bbb and frak and \cyclic
\usepackage{theorem}  % For \theorembodyfont to define
\def\mystretch{1.2}
\renewcommand{\baselinestretch}{\mystretch}

\ifx\pdfoutput\undefined
\def\pdfoutput{0}
\fi \ifcase\pdfoutput
\IfFileExists{srcltxx.sty}{\usepackage[active]{srcltxx}}%
{\IfFileExists{srcltx.sty}{\usepackage[active]{srcltx}}{}} \else
\IfFileExists{pdfsync.sty}{\usepackage{pdfsync}}{} \fi

\newtheorem{theorem}{Theorem}
\newtheorem{lemma}{Lemma}
\newtheorem{proposition}{Proposition}
\newtheorem{corollary}{Corollary}

{\theorembodyfont{\normalfont\rmfamily}
\newtheorem{definition}{Definition}
\newtheorem{remark}{Remark}

} \makeatletter %%
\newlength{\blackboxsize} \blackboxsize=1.2ex
\def\blackbox{\hbox{\vrule height  \blackboxsize
width  \blackboxsize depth 0ex}}
\newenvironment{proof}[1][]%
 {\def\proof@temp{#1}\par\noindent
  \textsc{Proof}\ifx\proof@temp\@empty\else\ (#1)\fi\hspace{1em}}
 {~~\hfill\blackbox\par\vspace{.5\baselineskip}}
\makeatother

\makeatletter \@addtoreset{equation}{section}

\makeatother

\makeatletter
\def\operatorname#1{\mathop{\operator@font #1}\nolimits}%
\makeatother %%
\newcommand{\suchthat}{\mathop{\,\vert\,}}
\newcommand{\half}{{\hbox{\normalsize$\frac12$}}}
\newcommand{\C}{\mathbb{C}}
\newcommand{\N}{\mathbb{N}}

\renewcommand{\P}{\mathbb{P}}
\newcommand{\R}{\mathbb{R}}
\newcommand{\Z}{\mathbb{Z}}

\renewcommand{\a}{\mathfrak{a}}
\newcommand{\g}{\mathfrak{g}}
\newcommand{\gl}{\mathfrak{gl}}
\newcommand{\h}{\mathfrak{h}}
\renewcommand{\k}{\mathfrak{k}}

\newcommand{\m}{\mathfrak{m}}
\newcommand{\n}{\mathfrak{n}}
\newcommand{\p}{\mathfrak{p}}

\renewcommand{\sl}{\mathfrak{sl}}
\newcommand{\so}{\mathfrak{so}}
\renewcommand{\sp}{\mathfrak{sp}}
\newcommand{\Sp}{Sp}
\renewcommand{\u}{\mathfrak{u}}
\newcommand{\su}{\mathfrak{su}}

\newcommand{\Ad}{\operatorname{Ad}}

\newcommand{\End}{\operatorname{End}}
\newcommand{\id}{\operatorname{Id}}
\newcommand{\Id}{\operatorname{Id}}
\newcommand{\Ker}{\operatorname{Ker}\,}
\renewcommand{\Im}{\operatorname{Im}\,}
\newcommand{\Span}{\operatorname{Span}}
\newcommand{\Symp}[1]{Sp(\R^{2(#1)},\Omega)}

\newcommand{\T}{\strut^\tau\!}
\newcommand{\Tr}{\operatorname{Tr}}

\let\ccirc\circ
\def\circ{{\ensuremath\ccirc}}

\newcommand{\cyclic}{\mathop{\kern0.9ex{{+}
 \kern-2.15ex\raise-.25ex\hbox{\Large\hbox{$\circlearrowright$}}}}\limits}

\let\ul\underline
\let\ol\overline

\let\ch\cosh \let\sh\sinh

\newcommand{\notationeq}{\stackrel{\hbox{\small{def}}}{=}}

\makeatletter
\newcount\hour\newcount\minute
\hour\time \divide\hour 60 \minute-\hour \multiply\minute 60
\advance \minute \time
\def\hours{\two@digits\hour}
\def\minutes{\two@digits\minute}
\def\Day{\two@digits\day}\def\Month{\two@digits\month}
\makeatother

\begin{document}\renewcommand{\baselinestretch}{1}

%%
%% Top matter
%%
\title{Transitive Subgroups of Transvections Acting on Some Symplectic
Symmetric Spaces of Ricci Type} %%
\author{~
Michel~Cahen$^1$, Simone Gutt$^{1,2}$, Amin D.  Malik$^1$ and John Rawnsley$^3$\\
\small {\ttfamily mcahen@ulb.ac.be, sgutt@ulb.ac.be, damin@ulb.ac.be, J.Rawnsley@warwick.ac.uk}\\[15pt]
\small $^1$ D\'{e}partement de Math\'{e}matiques\\[-5pt]
\small Universit\'{e} Libre de Bruxelles\\[-5pt]
\small Campus Plaine, CP 218\\[-5pt]
\small Boulevard du Triomphe\\[-5pt]
\small BE -- 1050 Bruxelles\\[-5pt]
\small Belgium\\[10pt]
\small $^2$ D\'{e}partement de Math\'{e}matiques\\[-5pt]
\small Universit\'{e} de Metz\\[-5pt]
\small Ile du Saulcy\\[-5pt]
\small F -- 57045 Metz Cedex 01\\[-5pt]
\small France\\[10pt]
\small$^3$  Mathematics Institute\\[-5pt]
\small Zeeman Building\\[-5pt]
\small University of Warwick\\[-5pt]
\small Coventry\ \ CV4\ \ 7AL\\[-5pt]
\small United Kingdom
~} %%

\date{~\\[10pt]{\small
File: {\sf \jobname.tex}\\[2pt]
\TeX{ed}: {\sf\number\year-\Month-\Day~\hours:\minutes}\\[2pt]
Last edit: {\sf\lastedit}  }}
 \maketitle

\setcounter{page}{0}

\thispagestyle{empty}

\begin{abstract}
Symmetric symplectic  spaces of Ricci type are a class of symmetric
symplectic spaces which can be entirely described by reduction of certain
quadratic Hamiltonian systems in a symplectic vector space. We determine,
in a large number of cases, if such a space admits a subgroup of its
transvection group acting simply transitively. We observe that the simply
transitive subgroups  obtained are one dimensional extensions of the
Heisenberg group.
\end{abstract}

%%
%% To generate a Table of Contents
%%
%\newpage
%\tableofcontents\thispagestyle{empty}

\renewcommand{\baselinestretch}{\mystretch}

\newpage

%***********************************************************************%
\section{Introduction}\label{section:one} 
This paper is devoted to the study of a class of symmetric symplectic
manifolds; those whose canonical connection is of Ricci type. More
precisely, we address the question of determining which of those
manifolds admit a simply transitive subgroup of transvections. This very
particular problem is motivated by quantization; in this case,  one has
at one's disposal  a number of techniques to construct  an invariant
quantization, either formal or convergent (see, for instance,
\cite{bib:AN1,bib:AN2}). When such a transitive subgroup $H$ exists, there is
another isomorphic one which is the image of $H$ by the involutive automorphism
$\sigma$ of the transvection group $G$  ($\sigma$ is conjugation in $G$ by the symmetry 
at a base point of the space). Observe that the  transvection group is generated by $H$ and
$\sigma(H)$; in particular, any object which is invariant under $H$ and
a symmetry will be automatically invariant under the whole of $G$.

The choice of these symmetric symplectic manifolds with Ricci-type
connections has two reasons.  The first one is  that it is one of the
few classes of symmetric symplectic  manifolds which is known and
completely classified. The second is that these symplectic manifolds
with connection are in some sense the analogue in symplectic geometry of
the classical space forms of Riemannian geometry.

A simply connected symmetric symplectic manifold of dimension $2n ~(n\ge
2)$ of Ricci type is entirely determined by the conjugacy class of a
non-zero  element $A$ of the symplectic Lie algebra
$\sp(\R^{2(n+1)},\Omega)$ such that $A^2=\mu \Id.$ More precisely,
connected symmetric symplectic spaces of Ricci-type of dimension $2n~(n\ge
2)$ are quotients of (the universal cover of) (the connected component of)
model manifolds obtained by reduction from the standard symplectic vector
space $(\R^{2(n+1)},\Omega)$ in the following way. Let  $A$ be a non-zero
element of $\sp(\R^{2(n+1)},\Omega)$ (with $n>1$) such that $A^2=\mu \Id$
and such that
\[
    \Sigma_A=\{x\in {\R}^{2(n+1)}\suchthat \Omega(x,Ax)=1\}\neq\emptyset.
\]
The 1-parameter group $\{\exp tA\}$ stabilizes $\Sigma_A$ and one considers
the space $M_A$  of orbits of the group $\{\exp tA\}$ and the canonical projection
\[
\pi : \Sigma_A \rightarrow M_A=\Sigma_A / \{ \exp tAÊ\}.
\]
The space $M_A$ has a manifold structure such that $\pi$ is a smooth
submersion, and is naturally endowed with a``reduced''  symplectic
structure and a ``reduced'' symplectic connection which is the canonical
connection for a natural ``reduced'' symmetric space structure.

Our main result can be described in terms of this characteristic element.

\begin{theorem} 
Let $0\neq A \in \sp(\R^{2(n+1)},\Omega)$ be   such that
$\Sigma_A\neq\emptyset$ and $A^2=\mu \Id$.
\begin{enumerate}
\item\label{musup0} If $\mu>0$, the space of orbits $M_A$ is diffeomorphic
to $TS^n$ (hence is always connected and simply connected when $n>1$) and
never admits a simply transitive subgroup of the transvection group.
\item\label{muinf0} If $\mu<0$, the symmetric bilinear form
$g(X,Y):=\Omega(X,AY)$ has $2p$ positive eigenvalues with $1\le p \le n+1;$
 the symmetric space  $M_A$  is connected and simply connected; it is
diffeomorphic to $\C^n$ if $p=1$ and to a complex vector bundle of rank
$n+1-p$ over $\P^{p-1}(\C)$ if $p>1$. It admits a simply transitive
subgroup of the transvection group if and only if $p=1$; non-isomorphic
simply transitive subgroups arise in this case.
\item\label{mueg0} If $\mu=0$, let $p ~(1\le p\le n+1)$ be the rank of $A$
and let $q ~(1\le q\le p)$ be the number of positive eigenvalues of the
symmetric bilinear form $g(X,Y):=\Omega(X,AY).$ The symmetric space $M_A$
is diffeomorphic to $T(S^{q-1}\times {\R}^{p-q})\times {\R}^{2(n+1-p)}$. If
$p=1$, each of the two connected components of $M_A$ is diffeomorphic to
the flat symplectic vector space and hence admits a simply transitive
subgroup of the transvection group. If $q\ne 1, 2 $ or $4,$ $M_A$ does not
admit a simply transitive subgroup of the transvection group. If $p=2$,
 $M_A$ admits a simply transitive subgroup of
the transvection group if and only if $q=1.$
\end{enumerate}
\end{theorem}

Let us mention that all simply transitive subgroups obtained here are 
$1$-dimensional extensions of the Heisenberg group of dimension $2n-1$.
The case of a sovable transvection group, i.e.when $A^2=0$ and $p=2$, 
was considered in 2006 in the m\'emoires de DEA of Amin Malik and 
Yannick Voglaire and was the framework of the quantization scheme developed by Bieliavsky in \cite{bib:AN2}.

The paper is organised as follows: in Section \ref{section:two}, we recall
definitions and known results concerning symplectic symmetric spaces of
Ricci-type. In particular, we describe  the manifold structure and the
transvection group of each $M_A$. Section \ref{section:3} is a summary of 
the facts about the topology of Lie groups which we need in the later
sections. We prove property \ref{musup0}  in Section \ref{section:musup0},
property \ref{muinf0} in Section \ref{section:muinf0} and property
\ref{mueg0} in Section \ref{section:mueg0}.

%***********************************************************************%
\section{Symmetric symplectic spaces whose curvature is of Ricci
type}\label{section:two}
We recall in this section known results which can be found for instance in 
the review \cite{bib:Ghent}.

\begin{definition}
A smooth symplectic manifold $(M,\omega)$ is a {\bf symmetric symplectic
space}  if there are symmetries, that is if there exists a smooth map
\[
s:M\times M\to M\
(x,y)\mapsto s(x,y)=s_xy
\] such that each ``symmetry''  $s_x$ squares to
the identity ($s_x^2=\id_{M}$), has $x$ as  isolated fixed point,  and
is a symplectic diffeomorphism ($s_x^*\omega=\omega$), such that
$s_xs_ys_x=s_{s_xy} \,\forall x,y\in M$. The {\bf canonical symmetric
connection}, $\nabla$, is defined by 
\[ 
\omega_x(\nabla_XY,Z) = \frac12 X_x\omega(Y+s_{x_*}Y,Z),\quad
    X,Y,Z\in \chi(M);
\]
it is symplectic (i.e. torsion-free and such that $\nabla\omega=0$) and
invariant by all symmetries.

A symmetric symplectic space $(M,\omega,s)$ is a homogeneous space
$M=G/K$ where $G$ is the group generated by products of even number
of symmetries ($G$ is called {\bf the transvection group}) and $K$
is the stabilizer in $G$ of a certain point $o$ (chosen as base
point).
\end{definition}
Symmetric symplectic spaces are not classified in general, but the
following cases are known:
\begin{enumerate}
  \item% [i)] 
  those whose transvection group $G$ is semisimple \cite{bib:sstransgrp};
  \item% [ii)] 
  those whose curvature of the canonical connection is
  of Ricci type (see  definition \ref{def:Riccitype} below) \cite{bib:CMF};
  \item% [iii)] 
  a class of symplectic symmetric spaces with nilpotent
  transvection group appearing as symmetric subspaces of a
  symplectic vector space \cite{bib:extrinsic};
  \item% [iv)] 
  symmetric symplectic spaces of dimension 2 and 4
  \cite{bib:sstransgrp}.
\end{enumerate}
\begin{definition}
Let $(V,\nu)$ be a symplectic vector space of dimension $2n$. An {\bf
algebraic symplectic curvature tensor} $R$ on $(V,\nu)$ is an element of
$\Lambda^2V^*\otimes S^2V^*$ such that
\[
    \cyclic_{X,Y,Z}R(X,Y,Z,T)=0
\]
where $  \cyclic_{X,Y,Z}$ denotes the sum over cyclic permutations of $X,Y$
and $Z$.
\end{definition}
The space $\mathcal{R}$ of algebraic symplectic curvature tensors on
$(V,\nu)$ splits (if $n\geq2$) into two subspaces which are irreducible
under the natural action of the symplectic group $Sp(V,\nu)$,
\cite{bib:Vaisman}. One writes:
\[
    \mathcal{R}=\mathcal{E}+\mathcal{W}.
\]
The $\mathcal{E}$ component of the curvature tensor $R$ can be
expressed in terms of the Ricci tensor $r$ associated to $R$. If
$R(X,Y)$ is the endomorphism of $(V,\nu)$ defined by
\[
    R(X,Y,Z,T)=\nu(R(X,Y)Z,T)
\]
the Ricci curvature associated to $R$ is the element of $S^2V^*$
\[
    r(X,Y)=\Tr(Z\mapsto R(X,Z)Y).
\]
The $\mathcal{E}$ component of the curvature $R$  has the form
\begin{eqnarray*}
  E(X,Y,Z,T) &=& \frac{-1}{2(n+1)}[2\nu(X,Y)r(Z,T)+\nu(X,Z)r(Y,T)
  +\nu(X,T)r(Y,Z)\\
   &~& \qquad\mbox{}-\nu(Y,Z)r(X,T)-\nu(Y,T)r(X,Z)].
\end{eqnarray*}

\begin{definition}\label{def:Riccitype}
Let $(M,\omega,\nabla)$ be a smooth symplectic manifold of
dimension $2n(n\geq2)$ endowed with a symplectic connection
$\nabla$. The connection $\nabla$ is said to be of {\bf Ricci type}
if, $\forall x\in M$, the curvature is of the form $R_x=E_x$.
\end{definition}
The following lemma is a direct consequence of the definitions.

\begin{lemma} Let $(M,\omega,\nabla)$ be a connected smooth
symplectic manifold of dimension $2n,n\geq2$ endowed with a smooth
symplectic connection $\nabla$ of Ricci type. Then the curvature
endomorphism $R(X,Y)$ is given by:
\begin{equation}\label{Ricci}
R_x(X,Y) = {\textstyle{\frac1{2n+2}}}(2\omega_x(X,Y)\rho_x
 + \rho_x Y\otimes \ul{X}+ X\otimes \ul{\rho_x Y}
 - \rho_x X\otimes \ul{Y}-Y\otimes \ul{\rho_x X})
\end{equation}
where $\rho$ is the Ricci endomorphism, i.e.
\[
    \omega(X,\rho Y)=r(X,Y)
\]
and where, for $X\in\chi(M)$, $\ul{X}=\omega(X,\,\cdot\,)$. 
Furthermore, there
exists a vector field $U$ such that 
\begin{equation}\label{U}
    \nabla_X\rho=\frac{-1}{2n+1}(X\otimes\ul{U}+U\otimes\ul{X});
\end{equation}
there exists a function $f$ such that:
\[
    \nabla_XU=-\frac{2n+1}{2(n+1)}\rho^2X+fX;
\]
and there exists a constant $K$ such that
\begin{equation}\label{K}
    \Tr\rho^2+\frac{4(n+1)}{2n+1}f=K.
\end{equation}
\end{lemma}

\begin{corollary} Two Ricci type connections on a real analytic 
connected symplectic manifold $(M,\omega)$ coincide if they have the same values 
$(\rho_o,U_o)$ of the Ricci endomorphism and of the
vector field $U$ at a point $o$ of $M$ (with $U$ defined be \ref{U}) and the same constant $K$ defined by
\ref{K}. 
\end{corollary}

\begin{corollary}\label{2.3} Let $(M,\omega,s)$ be a connected symplectic
symmetric space of dimension $2n$ whose canonical connection is of
Ricci type. Then the vector field $U$ vanishes; the function $f$ is
a constant,
$
f=\frac{2n+1}{4(n+1)^2}K
$,
and the Ricci endomorphism is such that
\[
    \rho^2=\frac{K}{2(n+1)}\Id.
\]
Furthermore given the value $\rho_o$ of $\rho$ at a base point $o$
of $M$, the canonical connection $\nabla$ of $(M,\omega,s)$ is
uniquely determined.
\end{corollary}
%%%%%%%%%%%%%

Let $(M_i,\omega_i,s_i),i=1,2$ be connected, simply-connected,
symmetric symplectic spaces of Ricci-type; let $o_i\in M_i$ and let
$\Psi:T_{o_1}M_1\to T_{o_2}M_2$ be a linear isomorphism such that
$\Psi^*\omega_2=\omega_1$ and $\Psi\circ\rho_1=\rho_2\circ\Psi$,
where $\rho_i$ is the Ricci endomorphism of the canonical connection
of $(M_i,\omega_i,s_i)$. Then $\Psi$ extends to a global symplectic
diffeomorphism $\widetilde{\Psi}:(M_1,\omega_1)\to(M_2,\omega_2)$
such that $\widetilde{\Psi}\,\circ\, s^{(1)}_{x}=
s^{(2)}_{\tilde{\Psi}(x)}\circ \,\widetilde{\Psi}$, i.e.
these two symmetric spaces are isomorphic. In view of Corollary
\ref{2.3}, we have an injective map of the set of isomorphism
classes of connected, simply-connected symmetric symplectic spaces
of Ricci-type and of dimension $2n$ into the set of conjugacy classes of
elements $\rho$ of the symplectic algebra $\sp(n,{\R})$ whose square is a
multiple of the identity (under the adjoint action of $Sp(n,{\R})$).
Indeed, recall that the Ricci endomorphism $\rho_o$ at a point $o$ of
$(M,\omega)$ is such that
\[
    r_o(X,Y)=\omega_o(X,\rho_oY)=r_o(Y,X)=\omega_o(Y,\rho_oX)
\]
and hence $\rho_o\in \sp(T_oM,\omega_o)$.
The examples which follow show that this map is also
surjective.

\bigskip

%%%%%%%%%%%%%
Let $(\R^{2(n+1)},\Omega)$ be the standard symplectic vector space of
dimension $2(n+1)$ and let $o\neq A$ be an element of
$\sp(n+1,{\R})$ such that $A^2=\mu \Id$ and such that
\[
    \Sigma_A=\{x\in {\R}^{2(n+1)}\suchthat \Omega(x,Ax)=1\}\neq\emptyset.
\]
Then $\Sigma_A$ is an embedded $2(n+1)$-dimensional submanifold of
${\R}^{2(n+1)}$. The group
$\exp tA$ stabilizes $\Sigma_A$ and has no fixed point in
$\Sigma_A$. The space $M_A$ of orbits of the group $\exp tA$ has a
shape depending on the sign of $\mu$:
\begin{itemize}
  \item[i)] if $\mu=k^2$, $k>0$, $\exp tA=\ch kt\ I+\frac{1}{k}\sh kt\
  A$;
  \item[ii)] if $\mu=-k^2$, $k>0$, $\exp tA=\cos kt\ I+\frac{1}{k}\sin kt\
  A$;
   \item[iii)] if $\mu=0$, $\exp tA= \Id+tA$.
\end{itemize}
%%%%%%%%%%%%%

\noindent$\bullet$ In the first case ($\mu=k^2$), $A$ admits $\pm k$ as
eigenvalues and
\[
    {\R}^{2(n+1)}=V^+\oplus V^-,\quad A|_{V^\pm}=\pm k\Id
\]
with $V^\pm$ lagrangian subspaces. In a basis $\{e_i,e'_i\suchthat1\leq i\leq
n+1\}$ adapted to this decomposition and such that
$\Omega(e_i,e'_j)=\delta_{ij}$, we can write
\[
x = x^++x^-, \quad\textrm{ so that }Ê\quad 
  \Omega(x,Ax) = -2k<x_+,x_->
\]
where $<,>$ is the standard scalar product on ${\R}^{n+1}$.
The orbit of $x_0=x_0^++x_0^-$ is
\[
    x(t)=e^{kt}x_0^++e^{-kt}x_0^-.
\]
An element  $x_0$ is in $\Sigma_A$ if and only if
\[
    <x_0^+,x_0^->=-\frac{1}{2k}
\]
and we  can choose a unique  element in its orbit, $x(t)$ such that
\[
    <x(t)^+,x(t)^->=-\frac{1}{2k},\quad <x(t)^+,x(t)^+>=1
\]
so that $M_A = \{(u,v)\suchthat<u,u>=1,<u,v>=-\frac{1}{2k}\}$ can be seen
as  an embedded submanifold. The map $\pi:\Sigma_A\to M_A$ given by
$(x_+,x_-)\mapsto(u,v):=(\frac{x_+}{<x_+,x_+>^{1/2}},<x_+,x_+>^{1/2}x_-)$
is a surjective submersion. The manifold  $M_A$ can be identified with
$TS^n = \{(u,w)\suchthat <u,u>=1 \,  <u,w>=0\}$ via
\[
    w=v+\frac{1}{2k}u.
\]

%%%%%%%%%%%%
\bigskip

\noindent$\bullet$ In the second case ($\mu=-k^2$), $A/k=J$ is a complex
structure compatible with $\Omega$, $\Omega(Jx,Jy)=\Omega(x,y)$, and the
corresponding bilinear form $g(x,y) = \Omega(x,Ay)$ has signature
$(2p,2q)$ where $p+q = n+1$ (if $g(x,x) \ne 0$, then $g(x,x)$ and
$g(Ax,Ax)$ have the same sign). Thus there exists a basis $\{e_j,f_j
\suchthat 1 \le j \le n+1\}$ and a diagonal matrix $\Id_{pq}=
\mathrm{diag}
(\epsilon_1,\dots,\epsilon_{n+1})$, where $\epsilon_r = 1$, $r \le p$
and $\epsilon_r = -1$, $r > p$ such that
\[
\Omega(e_i,f_j) = \epsilon_i\delta_{ij}, \qquad A e_l = k f_l, \qquad A
f_l = -k e_l.
\]
The hypersurface $\Sigma_A = \{u \suchthat \Omega(u,Au) = 1\}$ has equation
\[
x\cdot\Id_{pq} x + y\cdot\Id_{pq} y = 1
\]
if $u= \sum_i x^i e_i + y^i f_i$, $x=(x^1,\dots,x^{n+1})$,
$y=(y^1,\dots,y^{n+1})$. To have $\Sigma_A$ non-empty we must have
$p \ge 1$.

The group $\exp tA$ acts
by
\[
x \mapsto \cos kt\, x - \sin kt \, y,
\qquad y \mapsto \sin kt \, x + \cos kt \, y.
\]
If $p=1$, $\Sigma_A$ is diffeomorphic to $S^1\times\R^{2n}$ and the
quotient manifold is diffeomorphic to $\C^{n}$. If $1<p<n+1$, $\Sigma_A$
is diffeomorphic to $S^{2p-1} \times \R^{2q}$ and the quotient manifold
$M_A$ can be identified with a complex vector bundle of rank $q$ over
$\P_{p-1}( \mathbb{C})$. If $p=n+1$, $\Sigma_A$ is diffeomorphic to
$S^{2n+1}$ and the quotient manifold $M_A$ is $\P_n( \mathbb{C})$.

\bigskip

%%%%%%%%%%%%%%%%%

\noindent$\bullet$ In the third case ($\mu=0$), let $V=\Im A$; then $\Ker
A=(\Im A)^\bot\supset \Im A$. Let $p=\dim V,1\leq p\leq n+1$; let $W$ be
an arbitrary subspace of $\Ker A$, supplementary to $V$; then $W$ is
symplectic and $W^\bot=V\oplus V^*$, where $V^*$ is a lagrangian
subspace of $W^\bot$ in duality with $V$. Choose a basis $\{e_i, i\leq
p\}$ of $V$,$\{e_i^*, i\leq p\}$ of $V^*$ and $\{f_a,a\leq2(n+1-p)\}$ of
$W$ such that
\[
Ae_i^*=e_i,\quad  \Omega(e_i^*,e_j)=\epsilon_i\delta_{ij},\quad
\epsilon_i=\left\{\begin{array}{ll}1 
&\mbox{for } 1\le j\le q\\-1 &\mbox{for } q<j \le p\end{array}\right. .
\]
Denote by $\Omega^0$ the restriction of
$\Omega$ to the symplectic vector space $W: \quad
\Omega^0_{ab}=\Omega(f_a,f_b).$

If $u=\sum_{i=1}^p x^ie_i+\sum_{a=1}^{2(n+1-p)}
X^af_a+\sum_{j=1}^px_*^{j}e_j^*\notationeq x+X+x^*$, then
\[
    \Omega(u,Au)=\sum^p_{i=1}\epsilon_i(x_*^{i})^2.
\]
Since $q$ is the number of indices $i$ such that $\epsilon_i=1$, 
$\Sigma_A$ is not empty iff $1\leq q$.
The orbit of a  point $(x_o, X_0, x_{*0})$
under the action of $\exp tA$ has the form:
\[
\exp tA\cdot (x_o+X_o+x_{o*})=(x_o+tx_{o*},X_o,x_{o*})
=: (x_o(t),X_o(t),x_{o*}(t)).
\]
Hence, for any point in $\Sigma_A$:
\[
\sum^p_{i=1}\epsilon_ix^i(t)x_*^{i}(t)=\sum^p_{i=1}\epsilon_ix^i(0)x_*^{i}(0)+t
\]
so that each orbit of $\exp tA$ in $\Sigma_A$ contains a unique point
satisfying
\[
    \sum^p_{i=1}\epsilon_ix^i(t)x_*^{i}(t)=0, \qquad
  \left(\textrm{namely for }\ \   t=-\sum^p_{i=1}\epsilon_ix^i(o)x_*^{i}(0)\right).
\]
The projection $\pi:\Sigma\to M_A:(x,X,x_*)\mapsto
(x-(\sum^p_{i=1}\epsilon_ix^ix_*^{i})x_*,X,x_*)$ is a surjective submersion
on the space of orbits which is identified to the submanifold of
$\R^{2n+2}$ defined by
\[
    \sum^p_{i=1}\epsilon_ i(x_*^{i})^2=1, \qquad
    \sum^p_{i=1}\epsilon_ix^ix_*^{i}=0.
\]
Hence $M_A$  is diffeomorphic to $T(S^{q-1}\times {\R}^{p-q})\times
{\R}^{2(n+1-p)}$. If $q=1$, the manifold has two connected
components, each diffeomorphic to ${\R}^{2n}$. If $q=p$, $M_A$
is diffeomorphic to $TS^{p-1}\times {\R}^{2(n+1-p)}$.

%%%%%%%%%%%%
\begin{lemma} \label{lem 2.4} 
Let $(\R^{2(n+1)},\Omega)$ be the standard symplectic vector space of
dimension $2(n+1)$; let $o\neq A$ be an element of $\sp(n+1,\R)$ such
that
\begin{itemize}
  \item [i)] $A^2=\mu \Id$, $\mu\in \R$;
  \item [ii)] $\Sigma_A=\{x\in \R^{2(n+1)}\suchthat \Omega(x,Ax)=1\}\neq
  \emptyset$
\end{itemize}
then $M_A=\Sigma_A/\exp tA$ has a canonical structure of a smooth
manifold of dimension $2n$ and the canonical map $\Sigma\to
\Sigma/\exp tA$ is a smooth submersion. Furthermore
\begin{itemize}
\item if $\mu=k^2$, $M_A$ is diffeomorphic to $TS^n$, hence is connected
and simply connected for $n>1$;
\item if $\mu=-k^2$,  the signature of the symmetric bilinear form
$g(X,Y):=\Omega(X,AY)$ has $2p$ positive eigenvalues with $1\le p \le n+1;$
the symmetric space  $M_A$  is connected and simply connected; if $p=1$, it
is diffeomorphic to $\C^n$; if $1<p<n+1$  it is diffeomorphic to a complex
vector bundle of rank $n+1-p$ over $\mathbb{P}^{p-1}$; and if $p=n+1$, it
is diffeomorphic to $\mathbb{P}_n(\mathbb{C})$;
\item if $\mu=0$, let $p ~(1\le p\le n+1)$ be the rank of $A$ and let $q
~(1\le q\le p)$ be the number of positive eigenvalues in the symmetric
bilinear form $g(X,Y):=\Omega(X,AY).$ The space $M_A$ is diffeomorphic to
$T(S^{q-1}\times {\R}^{p-q})\times {\R}^{2(n+1-p)}$. If $q=1$, $M_A$ has
two connected components diffeomorphic to ${\R}^{2n}$. The space $M_A$ is
connected if $q>1$ and simply connected if $q>2.$
\end{itemize}
\end{lemma}

We now show  that on any such $M_A=\Sigma_A/\exp tA$, there exists a
symplectic structure $\omega$ and a connection $\nabla$ which is
symplectic, of Ricci type, and locally symmetric.

Let $\pi:\Sigma_A\to M_A=\Sigma_A/\exp tA$ be the natural projection; let
$x\in \Sigma_A$ and $y=\pi(x)$.
The tangent space to $\Sigma_A$ is
\[
    T_x\Sigma_A=\{Y\in T_x{\R}^{2(n+1)}\suchthat \Omega(Y,Ax)=0\}.
\]
Observe that $\forall x\in\Sigma_A$, the vector space ${\R}x$ is
transversal to $\Sigma_A$; we have an orthogonal decomposition:
\[
T_x{\R}^{2(n+1)}=T_x\Sigma_A\oplus\R x=\, \Span\{x,Ax\}\oplus \Span\{x,Ax\}^\bot.
\] 
 The subspace
$H_x = {} \Span\{ x,Ax\}^\bot$ is symplectic and $\pi_{*_x}:H_x\to T_yM_A$
is a linear isomorphism. Since $\exp tA:x\mapsto\exp tA.x$ maps
$H_x$ on $H_{\exp tAx}$, one may define a $2-$form on $M_A$ by
\[
\omega_y(X,Y)=\Omega_x(\ol{X},\ol{Y})
\]
where $X,Y\in T_yM_A, \ol{X},\ol{Y}\in H_x$ and
$\pi_*\ol{X}=X, \pi_*\ol{Y}=Y$. One checks readily that
this $2-$ form is indeed symplectic.
Let $\nabla^0$ be the standard flat symplectic connection on
$(\R^{2(n+1)},\Omega)$; define  (as in \cite{bib:BaguisCahen})
\[
    \ol{\nabla_XY}(x)=\nabla^0_{\ol{X}}\ol{Y}
    -\Omega(A\ol{X},\ol{Y})x+\Omega(\ol{X},\ol{Y})Ax.
\]
This is a linear connection on $M_A$, which is torsion free and has the
property that $\nabla\omega=0$ and hence is symplectic.

The curvature of this connection is given by:
\begin{eqnarray}\label{2.20}
  \ol{R(X,Y)Z} &=& -2\Omega(\ol{X},\ol{Y})A\ol{Z}-
  \Omega(\ol{X},\ol{Z})A\ol{Y}+\Omega(\ol{Y},\ol{Z})A\ol{X}\nonumber \\
    && \mbox{}+\Omega(A\ol{X},\ol{Z})\ol{Y}-\Omega(A\ol{Y},\ol{Z})\ol{X}.
\end{eqnarray}
Comparing \ref{Ricci} and \ref{2.20} one gets
\begin{eqnarray*}
  R(X,Y) &=& -\frac{1}{2(n+1)}[-2\omega(X,Y)(-2(n+1)\widetilde{A})
 + 2(n+1)\widetilde{A}Y\otimes \ul{X} \\
 && -2(n+1)\widetilde{A}X\otimes  \ul{Y}
    +X\otimes2(n+1) \ul{\widetilde{A}Y}-Y\otimes2(n+1) \ul{\widetilde{A}X}]
\end{eqnarray*}
where $\widetilde{A}(\in \End\ T_yM_A)$ is defined by
\[
    (\widetilde{A}X)_y=\pi_{*_x}(A\ol{X})_x.
\]
In particular this shows that the canonical connection is of Ricci
type and that the Ricci endomorphism is:
\[
    \rho=-2(n+1)\widetilde{A}.
\]
Since $\widetilde{A}\nabla_XY=\nabla_X\widetilde{A}Y$, we have
$\nabla_X\rho=0$ and hence the spaces $M_A$ are locally symmetric.
%%%%%%%

\begin{lemma} Let $0\neq A$ be an element of $\sp(n+1,{{\R}})$ such that
$A^2=\mu \Id$ and $\Sigma_A\neq\emptyset$. Then the orbit
space $M_A=\Sigma_A/\exp tA$ is a locally symmetric symplectic
manifold of dimension $2n$.
\end{lemma}

\begin{lemma} The manifold $M_A=\Sigma_A/\exp tA$ is a globally
symmetric symplectic manifold and its canonical connection is $\nabla$.
\end{lemma}
\begin{proof} Let $x\in \Sigma_A$ and let $y\in \R^{2(n+1)}$; if
\[
S_xy:=-y+2\Omega(y,Ax)x-2\Omega(y,x)Ax.
\]
Then $S_x$ belongs to $Sp(n+1,{{\R}})$ and commutes with
$A$, hence stabilizes $\Sigma_A$. If $\pi:\Sigma_A\to M_A$ is the
canonical projection, define:
\[
    s_{\pi(x)}\pi(y)=\pi(S_xy) \quad y\in\Sigma_A.
\]
This is well defined as the right hand side does not depend on the
choice of $x $(resp.{} $y$) in the fibre over $\pi(x)$ (resp.{}
$\pi(y)$). The diffeomorphism $s_{\pi(x)}$ of $M_A$ is symplectic and
one checks that it is affine (for the canonical connection $\nabla$
on $M_A$). Hence the conclusion.
\end{proof}

\begin{theorem} 
Let $(N,\nu,s)$ be a symmetric symplectic space of
dimension $2n$, $n\geq2$, whose curvature is of Ricci type. Then there
exists $A\neq 0$ in $\sp(n+1,\R)$ such that
$A^2=\mu \Id$ and  $\Sigma_A=\{x\in
  \R^{2n+2}\suchthat \Omega(x,Ax)=1\}\neq\emptyset$ and $(N,\nu,s)$
  is locally isomorphic to the symmetric symplectic space
  $(M_A=\Sigma_A/\exp tA,\omega,s)$.
  
If $(N,\nu,s)$ is connected, then its universal cover is globally
isomorphic to (the universal cover of) (a connected component of) $M_A$.
\end{theorem}
\begin{proof} Choose $o\in N$ and ${\xi}_o$ a symplectic frame of
$T_oN$ at $o$. Let $\check{\rho}$ be the element of $\sp(n,\R)$
defined by the Ricci endomorphism $\rho$ of $N$ at $o$. Let
$j:\R^{2n}\to (\R^{2(n+1)},\Omega)$ be the symplectic embedding of
$\R^{2n}$ onto the hyperplane spanned by the $2n$ last basis
vectors $e_i,1\leq i\leq 2n$. One assumes
$\Omega(e_0,e_i)=\Omega(e'_o,e_i)=0$ and $\Omega(e_0,e'_0)=1$. Define
$A\in \sp(n+1,\R)$ by:
\[
    A=\left(
        \begin{array}{ccc}
          0 & \frac{\mu}{4(n+1)^2} & 0 \\
          1 & 0 & 0 \\
          0 & 0 & -\frac{\check{\rho}}{2(n+1)} \\
        \end{array}
      \right)
\]
where $\check{\rho}=\mu \Id$. Then the manifold $\Sigma_A/\exp tA$ is
locally isomorphic to $(N,\nu)$.
\end{proof}
We end this paragraph by determining the transvection algebra of
the symmetric space $M_A=\Sigma_A/\exp tA$.

\begin{remark}
Any linear symplectic endomorphism $g\in\Sp(\R^{2(n+1)},\Omega)$ which
commutes with $A$  induces a symplectic diffeomorphism $\alpha(g)$ of
$M_A$, which is an affine map for $\nabla$ through
\[
\left(\alpha{(g)}\right)\left(\pi(x)\right):=\pi(gx), \quad x\in \Sigma_A
\]
The action of $G_1:=\{\, g\in Sp(\R^{2(n+1)},\Omega)\,\vert\, gA=Ag\,\}$
via $\alpha$ on $(M,\omega)$ is strongly Hamiltonian, and the function
$\tilde{f}_B$ on $ M_A$ corresponding to an element $B\in
\sp(\R^{2(n+1)},\Omega)$ is defined through
\[
\left(\pi^*(\tilde{f}_B)\right)(x) = 
\half \Omega(x,Bx) \quad \mbox{for } x\in \Sigma_A\subset \R^{2(n+1)}.
\]
\end{remark}

\begin{lemma}
The kernel of the homomorphism $\alpha$ from $G_1=\{g\in
Sp(n+1,\R)\suchthat gAg^{-1}=A\}$ into the group of affine symplectic
diffeomorphisms of $M_A$ is given by
\[
    \Ker\alpha=\{\exp tA\suchthat t\in \R\}.
\]
\end{lemma}

\begin{proof}
Consider an element $B\in Sp(n+1,\R)$ so that $BA=AB$ and assume
that $\alpha(B)=\Id_{M_A}$.
This means that for any $x\in \Sigma_A$ there exists $t_x \in \R$
so that $Bx=\exp{t_xA}\,x$. 
There exists a basis $u_i$, $i \le 2(n+1)$, of $\R^{2(n+1)}$ such that
$u_i \in \Sigma_A$ for all $i$. If $B \in Sp(n+1,\R)$ and $BA=AB$, there
exist $\tau_i \in \R$ such that
\[
B u_i = e^{\tau_iA}u_i
\]
then $\omega(u_i,u_j) = \omega(Bu_i,Bu_j) =
\omega(e^{(\tau_i-\tau_j)A}u_i,u_j)$; that is $e^{(\tau_i-\tau_j)A}u_i =
u_i$ and hence \\$e^{\tau_iA}u_i=e^{\tau_1A}u_i\, \forall i$. Thus 
$B=e^{\tau_1A}$.
\end{proof}

The group $\alpha(G_1)$ being transitive on $M_A$ and stable by
conjugation by the symmetry at a base point $o=\pi(x_o)$ of $M_A$
contains the transvection group of $M_A$, denoted by $G(M_A)$.

If $\widetilde{\sigma}$ denotes the involutive automorphism of
$G(M_A)$, $\widetilde{\sigma}(h)=s_ohs_o$ and
$\sigma:=\widetilde{\sigma}_{*_1}$ the corresponding
involutive automorphism of $\g$ (the Lie algebra of
$G(M_A)$), one knows that $\g$ is generated by the subspace
$\p$ of $\g$:
\[
\p=\{x\in \g\,\vert\,\sigma X=-X\}, \qquad 
\g=\p\oplus [\p,\p]=:\p\oplus  \k.
\]
The group $G_1$ is stable by conjugation by $S_{x_o}$. Denote by
$ \g_1$  the Lie algebra of $G_1$,  by $\widetilde{\sigma_1}$ 
the automorphism of $G_1$
defined by 
$\widetilde{\sigma_1}g:=
S_{x_o}gS_{x_o}$ and by $\sigma_1:=\widetilde{\sigma_1}_{*_1}$.
Introduce $\p_1:=\{Y\in \g_1\,\vert\, \sigma_1Y=-Y;\}$.
Consider  $\p_1+[\p_1,\p_1]\subset \g_1$.
Remark that  $\alpha_*\p_1=\p$ and
$\alpha_{*_{|\p_1}}$ is injective, since $\Ker\alpha_*=\R A$
and $\sigma_1A=A$. Hence
\begin{lemma}
With the notations defined above, the algebra $\g$ of the transvection
group of $M_A$ is isomorphic to $\p_1+[\p_1,\p_1]$ if $A\notin [\p_1,\p_1]$
and to $\p_1\oplus [\p_1,\p_1]/\R A$ if $A\in [\p_1,\p_1]$.
\end{lemma}
We determine this case by case.

\bigskip

%%%%%%%%%%%%%%%%%%%%%%
If $A^2=k^2 \Id$, $k>0$, we have, as indicated above, a basis
$\{e_i;i\leq n+1,e'_i;i\leq n+1\}$ such that in this basis:
\[
    \Omega=\left(
             \begin{array}{cc}
               0 & I_{n+1} \\
               -I_{n+1} & 0 \\
             \end{array}
           \right), \qquad
    A=\left(
          \begin{array}{cc}
            kI_{n+1} & 0 \\
            0 & -kI_{n+1} \\
          \end{array}
        \right).
\]
The algebra $\g_1=\{X\in \sp(n+1,\R)\suchthat [X,A]=0\}$ is composed of
elements
\[
    X=\left(
        \begin{array}{cc}
          X_1 & 0 \\
          0 & -{}^\tau X_1 \\
        \end{array}
      \right),
    \qquad X_1\in \gl(n+1,\R)
\]
Choose as base point $x_o\in \Sigma_A$,
$x_o=-\frac{1}{\sqrt{2k}}e_1+\frac{1}{\sqrt{2k}}e'_1$. The symmetry
$S_{x_o}$ has matrix
\[
S_{x_o}=\left(
          \begin{array}{cc}
            I_{1,n} & 0 \\
            0 & I_{1,n} \\
          \end{array}
        \right)
\]
and thus
\[
    \sigma_1X=S_{x_o}XS_{x_o}=\left(
                                \begin{array}{cc}
                                  I_{1,n}X_1I_{1,n} & 0 \\
                                  0 & -{}^\tau I_{1,n}X_1I_{1,n} \\
                                \end{array}
                              \right)
\]
and $X\in \p_1$ if and only if
\begin{equation}\label{2.31}
    X_1=\left(
          \begin{array}{cc}
            0 & {}^\tau b \\
            a & 0 \\
          \end{array}
        \right)
   \ a,b\in \R^n
\end{equation}
Consider the algebra $\p_1\oplus
[\p_1,\p_1]=:\p_1\oplus \k_1$ so that 
\begin{equation}\label{2.32}
    \k_1=\left\{X=\left(
              \begin{array}{cc}
                X_1 & 0 \\
                0 & -{}^\tau X_1 \\
              \end{array}
            \right)
            \left|\,X_1=\left(
                   \begin{array}{cc}
                     {}^\tau ba'-{}^\tau b'a & 0 \\
                     0 & a\otimes {}^\tau b'-a'\otimes {}^\tau b\\
                   \end{array}
                 \right)
            a,b,a',b'\in \R^{n}\right.\right\}.
\end{equation}
The matrices $X_1$ corresponding to elements of $\k_1$ have
zero trace. From (\ref{2.31}) and (\ref{2.32}), one sees that
$\p_1 \oplus \k_1$ is isomorphic to $\sl(n+1,\R)$ and $\k_1$ to
$\gl(n,\R)$. As $A\notin \k_1$, $\p_1 \oplus \k_1=\sl(n+1,\R)$ is
the transvection algebra.

\bigskip

%%%%%%%%%%%%%%%%%%
If $A^2=-k^2 \Id$, $k>0$, we have  a basis 
$\{\widetilde{e}_j;j\leq2(n+1)\}$ such that in this basis
\[
    \Omega=\left(
             \begin{array}{cc}
               0 & I_{n+1} \\
               -I_{n+1} & 0 \\
             \end{array}
           \right),
    \qquad A=k\left(
           \begin{array}{cc}
             0 & -I_{p,q} \\
             I_{p,q} & 0 \\
           \end{array}
         \right).
\]
The algebra ${\g}_1=\{X\in \sp(n+1,\R)\suchthat [X,A]=0\}$ is composed
of elements
\[
    X=\left(
        \begin{array}{cc}
          X_1 & X_2 \\
          X_3 & -{}^\tau X_1 \\
        \end{array}
      \right)
%\]
~\mbox{ with }~
%\[
    \begin{array}{l}  {}^\tau X_2=X_2,\\ %\qquad
     {}^\tau X_3=X_3,\end{array}\qquad
       \begin{array}{l}  X_3=-I_{p,q}X_2I_{p,q},\\%\qquad
    {}^\tau X_1 I_{p,q}+I_{p,q}X_1=0\end{array}.
\]
Let us introduce on $\R^{2(n+1)}$ the complex
structure $J=\frac{1}{k}A$ and the hermitian form
\[
    h=g-i\Omega, \quad \mathrm{where}\quad   g(X,Y)=\Omega (X,JY).
\]
Identifying $\R^{2(n+1)}$ with $\C^{n+1}$ by $  z = u_1+iI_{p,q}u_2$ with
$u_1,u_2\in\R^{n+1}$ one has
\[
  h(z,z') = {}^\tau z I_{p,q}\ol{z}'
\]
and the element $X\in \g_1$ acts on $\C^{n+1}$ by
\[
    Xz=(X_1-iX_2I_{p,q})z.
\]
Thus $\g_1$ is isomorphic to $\u(p,q)$ and $Az=-ikz$. Choose as base point
$x_o=\frac{1}{\sqrt{k}}\widetilde{e_1}\in\Sigma_A$. The symmetry
$S_{x_o}$ has matrix
\[
    S_{x_o}=\left(
              \begin{array}{cc}
                I_{1,n} & 0 \\
                0 & I_{1,n} \\
              \end{array}
            \right)
\]
and thus
\[
    \sigma_1X=S_{x_o}XS_{x_o}=\left(
				\begin{array}{cc}
					I_{1,n}X_1I_{1,n} &  I_{1,n}X_2I_{1,n}\\
					I_{1,n}X_3I_{1,n} & -I_{1,n}{}^\tau X_1I_{1,n} \\
				\end{array}
			\right)
\]
and $X\in \p_1$ if and only if
\begin{eqnarray*}
  X_1 &=& \left(
            \begin{array}{cc}
              0 & -{}^\tau a I_{p-1,q} \\
              a & 0 \\
            \end{array}
          \right),
   \qquad a\in\R^n\\
  X_2 &=& \left(
            \begin{array}{cc}
              0 & {}^\tau c \\
              c & 0 \\
            \end{array}
          \right),
   \qquad c\in\R^n\\\
  X_3 &=& -I_{p,q}X_2I_{p,q}.
\end{eqnarray*}
One sees that
$\k_1=[\p_1,\p_1] $ is equal to $ \u(p-1,q) $ and that $
\p_1 \oplus \k_1 $  is equal to $\su(p,q)$. As $A\notin
\k_1$, the transvection algebra is $\mathfrak{g}=\su(p,q)$.

\bigskip

%%%%%%%%%%%%%%%

If $A^2=0$, $A\neq0$, we have a basis $\{e_i,i\leq
p;f_a,a\leq2(n+1-p);e^*_i,i\leq p\}$ of $\R^{2n+2}$ such that in this
basis
 \[
    \Omega=\left(
             \begin{array}{ccc}
               0 & 0 & -I_{q,p-q} \\
               0 & \Omega^0 & 0 \\
               I_{q,p-q} & 0 & 0 \\
             \end{array}
           \right)
    \ A=\left(
          \begin{array}{ccc}
            0 & 0 & I_p \\
            0 & 0 & 0 \\
            0 & 0 & 0 \\
          \end{array}
        \right)1\leq p\leq n+1,1\leq q\leq p
\]
where
\[
   \Omega^0_{ab}:= \Omega(f_a,f_b).
\]
The algebra ${\g}_1=\{X\in \sp(n+1,\R)\suchthat [X,A]=0\}$ is composed
of elements
\[
    X = \left(\begin{array}{ccc}X_1&X_2&X_3\\0&X_5&X_6\\
0&0&X_1\end{array}\right)
\]
where
\begin{eqnarray*}
  &&{}^\tau X_1I_{q,p-q} + I_{q,p-q} X_1= 0,
  \quad  {}^\tau X_5\Omega^0 + \Omega^0 X_5 = 0 \\
  &&{}^\tau X_2 =\Omega^0 X_6 I_{q,p-q},
  \quad  {}^\tau X_3I_{q,p-q} - I_{q,p-q}
X_3 = 0
\end{eqnarray*}
Choose as point $x_0\in\Sigma_A$, $x_0=e^*_1$. The symmetry
$S_{x_o}$ has matrix
\[
    S_{x_0}=\left(\begin{array}{ccc}
I_{1,p-1}&0&0\\
0&-I_{2(n+1-p)}&0\\
0&0&I_{1,p-1}
\end{array}\right).
\]
and thus
\[
    \sigma_1X=S_{x_o}XS_{x_o}=\left(\begin{array}{ccc}
I_{1,p-1}X_1I_{1,p-1}&-I_{1,p-1}X_2&I_{1,p-1}X_3I_{1,p-1}\\
0&X_5&-X_6I_{1,p-1}\\
0&0&I_{1,p-1}X_1I_{1,p-1}
\end{array}
\right).
\]
and hence $X\in \p_1$ if and only if
\begin{eqnarray*}
  X_1 &=& \left(
            \begin{array}{cc}
              0 & -\T p'' I_{q-1,p-q}\\
              p'' & 0 \\
            \end{array}
          \right)\qquad p'' \in\R^{p-1}
   \\
  X_6 &=& \left(
            \begin{array}{cc}
             P &0
            \end{array}
          \right),\qquad P\in\R^{2(n+1-p)}
   \\
  X_3 &=& \left(
            \begin{array}{cc}
              0 & \T p'I_{q-1,p-q} \\
              p' & 0 \\
            \end{array}
          \right),\qquad p'\in\R^{p-1}
   \\
  X_5 &=& 0.
\end{eqnarray*}
Notice that an element $X\in \p_1$  is determined by a triple of matrices
$(X_1,X_6,X_3)$, as specified above. An element of $[p_1,p_1]$ is also determined
by  matrices $(X_1,X_6,X_3)$ and 
\begin{eqnarray*}
    [(X_1,X_6,X_3),(X'_1,X'_6,X'_3)]&=&\nonumber\\
    && \mbox{}\kern-5cm \left([X_1,X'_1],X_6X_1'-X_6'X_1,[X_1,X_3']
    +[X_3,X_1']+{}^\tau(\Omega_0 X_6 I_{q,p-q})X_6'-{}^\tau(\Omega_0 X_6' 
    I_{q,p-q})X_6\right).
\end{eqnarray*}
The algebra $\p_1+[\p_1,\p_1]=\p_1+\k_1$ is thus composed of
elements
\[
    X = \left(\begin{array}{ccc}X_1&{}^\tau(\Omega_0 X_6 
     I_{q,p-q})&X_3\\0&0&X_6\\
0&0&X_1\end{array}\right)
\]
where
\[
X_1\in \so(q,p-q), \quad 
X_6 \in \gl(\R^p,\R^{2(n+1-p)}),\quad  {}^\tau X_3I_{q,p-q} - I_{q,p-q}
X_3 = 0.
\]
Observe that the subalgebra generated by the $X_6$'s and $X_3$'s is a
nilpotent ideal and that $A\in \k_1$. Recall that the transvection Lie
algebra is $\g=(\p_1+\k_1) /\R A$. Thus

\begin{lemma} \label{lem5} Let $A\in \sp(n+1,\R)$, $A\neq 0$, $A^2=0$;
assume $\Sigma_A\neq\emptyset$ and let $M_A=\Sigma_A/\exp tA$. Let $p
~(1\le p\le n+1)$ be the rank of $A$ and let $q ~(1\le q\le p)$ be the
number of positive eigenvalues in the symmetric bilinear form
$g(X,Y):=\Omega(X,AY).$ Then
\begin{itemize}
  \item [i)] if $p=1$ (and hence also $q=1$), $M_A$ has two connected
  components diffeomorphic to $\R^{2n}$ and the transvection algebra
  is the $2n$-dimensional abelian algebra;
  \item [ii)] if $p=2$, the transvection algebra is solvable and
  admits a codimension $1$ nilpotent ideal;
  \item [iii)] if $p>2$, the transvection algebra has a Levi factor
  isomorphic to $\so(q,p-q)$ and a nilpotent radical.
\end{itemize}
\end{lemma}
We shall now investigate the existence of a simply-transitive subgroup $H$
of the transvection group $G(M_A)$ for the symmetric spaces
$M_A=\Sigma_A/\exp tA$. We separate the discussion into three cases
$A^2=k^2I$, $A^2=-k^2I$ and $A^2=0$.

%%%%%%%%%%%%%%%%%
%%%%%%%%%%%%%%%%%%%%%%
\section{Topology of Lie Groups}\label{section:3}

In this section we summarize the facts about the topology of Lie groups
which we need in the later sections. All the results cited can be found
in the survey article \cite{bib:Samelson} by Hans Samelson and its
bibliography.

\begin{theorem}\label{thm:TopLieGp}
Let $H$ be a Lie group diffeomorphic to $S^a\times \R^b$ or $TS^a\times
\R^b$. Then $a=1$ or $a=3$ and $H$ has maximal compact subgroup
isomorphic to $S^1$ or $SU(2)$.
\end{theorem}

\begin{proof} 
Note that $TS^a\times\R$ is diffeomorphic to $S^a\times\R^{a+1}$ so we have
to consider, $S^a\times \R^b$ and $TS^a$.

By the Iwasawa--Malcev theorem, $H$ is diffeomorphic to a product of its
maximal compact subgroup $K$ and a euclidean space. If $H$ is diffeomorphic
to $S^a\times \R^b$ and $a\ge 2$, $H$ must be simply connected and hence so
is $K$. The latter is then a product of simply connected compact Lie groups
with simple Lie algebras. By a theorem of Hopf each factor has the same
real cohomology as a bouquet of odd dimensional spheres with one copy of
$S^3$ always occurring and with as many spheres as the rank of the Lie
algebra. Since $S^a\times\R^b$ has the real cohomology of a single sphere,
there can only be one simple factor and it must have rank 1. Thus $K$ is
isomorphic to $SU(2)$ and $a=3$. If $a=1$ then $K$ is compact, connected,
1-dimensional and so $K=S^1$.

Suppose now that $H$ is diffeomorphic to $TS^a$ then $H\times \R$ is
diffeomorphic to $S^a\times\R^{a+1}$ so by the previous result $a=1$ or
$a=3$ and $H\times\R$ has maximal compact subgroup $S^1$ or $SU(2)$ and
hence so has $H$. Since $S^1$ and $S^3$ have trivial tangent bundles, we
are done.
\end{proof}

%%%%%%%%%%%%%%%%%%%%%%%
%%%%%%%%%%%%%%%%%%%%%%%%
\section{Simply-transitive subgroups of the group of trans\-vections $G(M_A)$ for
$M_A=\Sigma_A/\exp tA$, when $A^2=k^2I$} \label{section:musup0}

We have shown
\begin{itemize}
  \item [i)] that $M_A=TS^n$ (Lemma \ref{lem 2.4});
  \item [ii)] that the transvection algebra of $M_A$ is isomorphic to
  $\sl(n+1,\R)$;
  \item [iii)] that in an appropriate basis of $\R^{2(n+1)}$,
  $\{e_i^+,i\leq n+1;e_i^-,i\leq n+1\}$
  \[
    \Sigma_A=\{x=\sum_i{x^+}^ie_i^++{x^-}^ie_i^-\suchthat 
    \sum_{i=1}^{n+1}x_i^+x_i^-=-\frac{1}{2k}\};
  \]
  \item [iv)] that the projection $\pi:\Sigma_A\to
  TS^n=\Sigma_A/\exp tA$ has the form
  \[
    \pi(x_+,x_-)=\left(\frac{x_+}{<x_+,x_+>^{1/2}}=u,<x^+,x^+>^{1/2}x_-
    +\frac{1}{2k}\frac{x_+}{<x_+,x_+>^{1/2}}=w\right).
  \]
\end{itemize}
The subgroup of $Sp(n+1,\R)$, $G_1=\{C\suchthat [C,A]=0\}$ is $Gl(n+1,\R)$.
It acts on $TS^n$ by
\[
    B.(u,w)=\left(\frac{Bu}{<Bu,Bu>^{1/2}}\, ,\,<Bu,Bu>^{1/2}\T
    B^{-1}(w-\frac{1}{2k}u)+\frac{1}{2k}\frac{Bu}{<Bu,Bu>^{1/2}}\right)
\]
and the kernel of effectivity is $\{B=\lambda I,\lambda>0\}$; the
connected subgroup $GL^{+}(n+1,\R)$ modulo the kernel of effectivity
is $SL(n+1,\R)$. Hence the transvection group of $TS^n$ is
$SL(n+1,\R)$ and the stabilizer of the point $(u=e_1^+,w=0)$ is the
group $GL^{+}(n,\R)$.

\begin{theorem}\label{thm 3.1}
No symplectic symmetric space  $M_A = \Sigma_A / \exp tA$, where $A^2 = k^2
I$, admits a subgroup of its transvection group acting simply-transitively.
\end{theorem}

Equivalently,

\begin{proposition}
There is no Lie subgroup of $SL(n+1,\R)$ which acts simply-transitively
on $TS^n$, $n \ge 1$.
\end{proposition}

\begin{proof}
If there is a Lie subgroup of $SL(n+1,\R)$ acting simply-transitively
then $TS^n$ is diffeomorphic to a Lie group and hence by Theorem
\ref{thm:TopLieGp} $n=1$ or $n=3$.
\begin{itemize}

 \item[(i)] $n=1$. $TS^1$ is diffeomorphic to $S^1\times\R$. If there
exists a subgroup $H$ of $SL(2,\R)$ acting simply-transitively on
$TS^1$, $H$ must be non-abelian since $SL(2,\R)$ has no 2 dimensional
abelian subgroups. $H$ is diffeomorphic to $TS^1$ so must have $S^1$ as
maximal compact subgroup. We consider the Lie algebra $\h$ of $H$ as a
representation of this circle subgroup. Real irreducibles of $S^1$ are
either the trivial one dimensional representation or non-trivial two
dimensional representations. Since $\h$ already contains one trivial
representation it must be the sum of two trivial representations and
hence be abelian, so we have a contradiction showing no such $H$ exists.
    
 \item[(ii)] $n=3$. 
This is the most complicated case and the rest of this section is
devoted to its examination.
\end{itemize}
 Let $H$ be a subgroup of $SL(4,\R)$ acting simply-transitively on
 $TS^3$. By Theorem \ref{thm:TopLieGp}, a maximal compact subgroup $K$
 of $H$ is isomorphic to $SU(2)$. Up to conjugation, we may assume
 $K\subset SO(4,\R)$. Let us investigate the $SU(2)$ subgroups of
 $SO(4,\R)$. For this, consider $\mathbb{H}$ the space of quaternions.
 We have a natural map
  \[
    (SU(2)\times SU(2))\times\mathbb{H}\to\mathbb{H}:((q_1,q_2),x)
    \mapsto q_1xq_2^{-1}
  \]
which is norm preserving:
\[
    (q_1xq_2^{-1})\ol{(q_1xq_2^{-1})}=x\ol{x}, \qquad
    \ol{x}=\textrm{quaternionic conjugate of $x$}
\]
and hence we have a homomorphism
\[
    SU(2)\times SU(2)\to O(4)
\]
which by connectedness takes it values in $SO(4)$. The kernel of
this homomorphism is $\mathbb{Z}_2=\{(1,1),(-1,-1)\}$. Hence
\begin{itemize}
  \item [a)] $SU(2)\times SU(2)/\mathbb{Z}_2$ is isomorphic to $SO(4)$;
  \item [b)] $SU(2)\times\{1\}$ and $\{1\}\times SU(2)$ project
  isomorphically into $SO(4)$.
\end{itemize}
We shall denote these subgroups $SU(2)_L$ and $SU(2)_R$. They are
normal subgroups of $SO(4)$, hence not conjugate. One checks that
they are conjugate in $SL(4,\R)$; indeed if $C$ denotes conjugation
in $\mathbb{H}$:
\[
    C\circ R_{q_2^{-1}}\circ C=L_{q_2}.
\]
Any other compact $3$-dimensional subgroup of $SO(4)$ must have a
non-trivial component in each of $SU(2)_L$ and $SU(2)_R$, and since
there are no $2$-dimensional subgroups of $SU(2)$ we can use the
projection into one factor to parametrise, and the other factor must
then be related by an (inner) automorphism from which we see that such a
subgroup must be conjugate to the image in $SO(4)$ of the diagonal
$SU(2) = \{(q,q)\suchthat q \in SU(2)\}$ and this projects to an $SO(3)$
subgroup.

Thus $SL(4,\R)$ has two conjugacy classes of compact $3$-dimensional
subgroups whose members are either isomorphic to $SU(2)$ (and we may
take $SU(2)_L$ as a representative) or isomorphic to $SO(3)$. Hence in
looking for subgroups $H$ diffeomorphic to $TS^3$ we may assume $H$ has
maximal compact subgroup $SU(2)_L$ without any loss of generality.

We now describe the adjoint action of $SU(2)_L$ on $\sl(4,\R)$. The Lie
algebra $\so(4)$ consists of all skew-symmetric matrices in $\sl(4,\R)$
and has an invariant $9$-dimensional complement $\p$ given by all
traceless symmetric matrices. These form an irreducible representation
of $SO(4)$ under the adjoint action (indeed if it were not one could
construct a non-trivial ideal in $\sl(4,\R)$) and hence of
$SU(2)_L\times SU(2)_R$. Since both factors act non-trivially (if it
were not, one could construct a non trivial ideal of $\sl(4,\R)$ in
$\so(4)$), $\p$ must be isomorphic to a tensor product of the
$3$-dimensional irreducible representation of each factor as $9 = 3
\times 3$ is the only non-trivial factorisation.  Thus as $\su(2)_L$
module, $\sl(4,\R)$ can be written as
\[
    \sl(4,\R)=\k\oplus
    \p=(\su(2)_L\oplus(\R\oplus\R\oplus\R))\oplus\R^3\oplus\R^3\oplus\R^3
\]
where $\R^3$ denotes the irreducible 3 dimensional $\su(2)_L$
module.

The algebra $\mathfrak{h}$ of $H$ contains $\su(2)_L$ and is stable by
the $\su(2)_L$ action. Hence
\[
    \mathfrak{h}=\su(2)_L\oplus \mathfrak{h}_1
\]
where $\mathfrak{h}_1$ is a 3 dimensional representation of
$\su(2)_L$. In view of the above, it is either the sum of three
trivial 1 dimensional representations or a 3 dimensional irreducible
representation. In the first case, $\mathfrak{h}$ would be isomorphic
to $\so(4)$ and $H$ can not be diffeomorphic to $TS^3$. Hence
$\mathfrak{h}_1$ is a 3 dimensional irreducible representation of
$\su(2)_L$. Thus
$\mathfrak{h}_1\subset\p$. Now $\p$ is isomorphic
as $\su(2)_L$ module to $\su(2)_L\otimes \mathcal{W}$, where
$\mathcal{W}$ is a trivial 3 dimensional $\su(2)_L$ module. We can
view the injection $\mathfrak{h}_1\to \p$ as an intertwining
map $\phi$ from $\su(2)_L$ into $\su(2)_L\otimes \mathcal{W}$. If
$v\in \su(2)_L$ and $e_i(i\leq3)$ is a basis of $\mathcal{W}$, we
have
\[
    \phi(v)=\sum_{i=1}^3\phi_i(v)\otimes e_i
\]
For any $v'\in \su(2)_L$:
\begin{eqnarray*}
  [v',\phi(v)] &=& \sum^3_{i=1}[v',\phi_i(v)]\otimes e_i=\phi([v',v]) \\
   &=& \sum_{i=1}^3\phi_i([v',v])\otimes e_i
\end{eqnarray*}
Hence each $\phi_i$ is an intertwining operator and by Schur's
lemma, is a multiple of the identity; hence
\[
    \phi(v)=\sum_{i=1}^3\lambda_iv\otimes
    e_i=v\otimes\sum_{i=1}^3\lambda_i e_i=:v\otimes w
\]
Thus $\mathfrak{h}_1=\su(2)_L\otimes w$.

Now the left action of $SU(2)_L$ on $\R^4(\sim\mathbb{H})$ has the
form
\begin{eqnarray*}
  qxq^{-1} &=&(q_0,q)(x_0,x)(q_0,-q)  \\
   &=& (x_0,(q_0^2-q^2)x+2q_0q\wedge
  x+2q_0x_0q)=:(x_0,R(q)x)
\end{eqnarray*}
where $q_0^2+|q|^2=1$. Similarly the map $x\mapsto qx$ has matrix
\[
q_L=\left(
      \begin{array}{cc}
        q_0 & -\T q \\
        q & q_0I+q_\wedge \\
      \end{array}
    \right)
\]
and the map $x\mapsto xq^{-1}$ has matrix
\[
    q_R=\left(
          \begin{array}{cc}
            q_0 & \T q \\
            -q & q_0I+q_\wedge \\
          \end{array}
        \right)
\]
Define $\eta:\R^3\otimes\R^3\to P$
\[
    \eta(x,y)=\left(
                \begin{array}{cc}
                  x.y & \T(y\wedge x) \\
                  y\wedge x & x\otimes\T y+y\otimes \T x-xy1 \\
                \end{array}
              \right)
\]
This is indeed in $\p$ as $\T \eta(x,y)=\eta(x,y)$ and
$tr\eta(x,y)=0$. Furthermore the map is clearly surjective. The
following relations show that the map $\eta$ exhibits the $SU(2)_L$
and $SU(2)_R$ action on $\p$
\begin{eqnarray*}
  q_L\eta(x,y)q_L^{-1} &=& \eta(R(q)x,y) \\
  q_R\eta(x,y)q_R^{-1} &=& \eta(x,R(q)y)
\end{eqnarray*}
To conclude, we observe that the element $t\T w\otimes w$ of
$\mathfrak{h}_1(t\in\R)$ corresponds to the element of $\p$
\[
    \eta(tw,w)=\left(
                 \begin{array}{cc}
                   t|w|^2 & 0 \\
                   0 & 2tw\otimes\T w-t|w|^21 \\
                 \end{array}
               \right).
\]
This belongs to the Lie algebra of the subgroup $GL^+(3,\R)$ which
is the stabilizer of the point $(u=e_1^+,w=0)$. Hence the orbit of
$H$ on $TS^3$ is at most of dimension 5; a contradiction.
\end{proof}

%%%%%%%%%%%%%%%%%%%%
%%%%%%%%%%%%%%%%%%%%%%
\section{Simply-transitive subgroups of the group of trans\-vections
$G(M_A)$ for $M_A=\Sigma_A/\exp tA$, when $A^2=-k^2I$}\label{section:muinf0}
 
We have shown that such spaces of dimension $2n$ are characterized by an
integer $p$, (the symmetric bilinear form $g(X,Y):=\Omega(X,AY)$ has $2p$
positive eigenvalues) $1\leq p\leq n+1$. If  $p=1$ $M_A=\C^n$ 
; $M_A$ is a complex vector bundle of rank $q=n+1-p$  over
$\mathbb{P}^{p-1}(\C)$ if $1<p<n+1$; and $M_A=\mathbb{P}^n(\C)$ if $p=n+1$.
All these spaces are simply-connected. We have also shown that the
transvection algebra is $\su(p,q)$ and that the isotropy algebra is
$u(p-1,q)$.

The subgroup of $\Symp{n+1}$ which commutes with $A$ acts on $z\in \C^{n+1} $ as
$U(p,q)$ and $\exp tA$ acts as multiplication by $e^{-ikt}$. Hence on
the quotient we obtain the projective pseudo-unitary group
$PU(p,q)=U(p,q)/e^{it}I_{n+1} = SU(p,q)/\Z_{n+1}$ as a symmetry group.
However $PU(p,q)$ is simple and acts faithfuly so it must coincide with the transvection
group of $M_A$. The stabiliser of $z=e_1$ is isomorphic to $U(p-1,q)$.

\begin{theorem} The symplectic symmetric space $M_A=\Sigma_A/\exp
tA$, $A^2=-k^2I$, admits a simply-transitive subgroup of its
transvection group if and only if $p=1$.
\end{theorem}
\begin{proof} 
To determine whether there is a simply-transitive subgroup of the
transvection group, we note that all the above manifolds are
simply-connected, and when $p>1$ they have the homotopy type of
$\P_{p-1}(\C)$ which has first non-vanishing cohomology group in degree 2.
However a simply-connected Lie group which is not contractible will have
the cohomology of a product of odd dimensional spheres, with at least one
copy of $S^3$ occurring. Thus the first non-zero cohomology group will be
in degree 3. Hence for $p>1$ there can be no simply-transitive subgroup of
the transvection group. Thus we have shown $p=1$ is necessary.

\bigskip

If $p=1$, $M_A=SU(1,n)/U(n)$ (since the kernel of effectivity is contained in
the $U(n)$ subgroup). Now $U(n)$ is the maximal compact subgroup of
$SU(1,n)$ and  we have an Iwasawa decomposition
\[
    SU(1,n)=U(n)AN
\]
where $A$ is a  real closed 1-parameter subgroup isomorphic to $\R$. The
group $AN$ acts simply-transitively on $M_A$.
\end{proof}

\begin{remark}
The subgroup $H=AN$ above is isomorphic to the extension
$K_{2n}$ of the $2n-1$-dimensional Heisenberg group $H_{2n-1}$ by
dilations; the Heisenberg Lie algebra $\h_{2n-1}$ is defined by
\[
\h_{2n-1}=\left\{ (X,a)\,\vert\, X\in \R^{2(n-1)},a\in\R\right\} 
\ \mathrm{with} \ \,[(X,a),(Y,b)]=\left(0,\Omega^0(X,Y)\right)
\]
for a non degenerate skewsymmetric $2$-form $\Omega^0$ on $\R^{2n-2}$;
and the Lie algebra of its extension by dilations is
\[
\mathcal{K}_{2n}=\h_{2n-1}\oplus \R D \ \mathrm{with} \  
D \left( (X,a)\right):=(X,2a).
\]
\end{remark}

The question is whether there are any subgroups other than $AN$ which
act simply-transi\-tively on $G/K$ in this case. That there are follows
from the following Proposition:

\begin{proposition} 
Let $G$ be a connected semisimple Lie group; let $\mathfrak{g}=\k+\p$ be a
Cartan decomposition of its Lie algebra. Let $K$ be the analytic subgroup
with algebra $\k$; let $\a$ be the maximal abelian subalgebra of $\p$; let
$A$ be the analytic subgroup with algebra $\a$. Let $\n$ be the subalgebra
of $\mathfrak{g}$ spanned by the positive root vectors corresponding to a
choice of positive roots of the pair $(\mathfrak{g},\a)$. Let $N$ be the
analytic subgroup with algebra $\n$. Let $\m$ be the centraliser of $\a$ in
$\k$; let $\varphi:\a\to\m$ be a homomorphism; let
$\a_\varphi=\{X+\varphi(X)\suchthat X\in\a\}$. Then $a_\varphi$ is an
abelian subalgebra of $\mathfrak{g}$ of the same dimension as $\a$. Let
$A_\varphi$ be the analytic subgroup with algebra $\a_\varphi$. Then the
map $K\times A_\varphi\times N\to G:(k,a_\varphi,n)\mapsto ka_\varphi n$ is
a global decomposition of $G$ as a product of three closed subgroups.
Finally, $H_\varphi=A_\varphi N$ is a closed solvable subgroup of $G$
acting simply-transitively on $G/K$.
\end{proposition}
\begin{proof} It is clear that $\a_\varphi$ is an abelian subalgebra
of $\mathfrak{g}$ of the same dimension as $\a$ and that
$[\a_\varphi,\n]\subset\n$; hence $\h_\varphi=\a_\varphi+\n$
is a solvable algebra. Let $g\in G$ and let
\[
    g=kan
\]
be its Iwasawa decomposition. Then, if $a=\exp H$,
\[
    g=k\exp-\varphi(H)\exp(H+\varphi(H))n=:k'a'n.
\]
Hence any $g$ can be written as a product of an element of $K$, an
element of $A_\varphi$ and an element of $N$. Uniqueness of the
decomposition comes from the uniqueness of the Iwasawa decomposition
and
\begin{eqnarray*}
  k_1a_1^\varphi n_1 &=& k_1\exp(H+\varphi(H))n_1=k_1\exp\varphi(H)\exp Hn_1\\
   &=& k_1\exp\varphi(H)a_1n_1
\end{eqnarray*}
and $k_1\exp \varphi(H)\in K$.\\
Let $(a'_i)$, $i\in \N$ be a sequence of elements of $A_\varphi$
converging in $G$; then $a'_i=\exp(H_i+\varphi(H_i))$, $H_i\in\a$;
then $(H_i)$ converges to $H$ as the projection on an Iwasawa factor
is continuous. Hence
$\lim_{i\to\infty}\exp(H_i+\varphi(H_i))=\exp(H+\varphi(H))$ and
$A_\varphi$ is closed. Thus also $H_\varphi$ is a closed solvable
subgroup of $G$.
\end{proof}

\begin{remark} 
The adjoint action of $\a_\varphi$ on $\n^\C$ is semi-simple. The
eigenvalues have an imaginary part determined by $\varphi$. Thus in general
distinct $\varphi$'s will give rise to non-isomorphic subgroups
$H_\varphi$. These subgroups are all $1$-dimensional extensions of $N$
which is isomorphic to the Heisenberg group.
\end{remark}

\begin{lemma} 
Let $H$ be a linear Lie group acting simply-transitively on $\R^n$; then
$H$ is solvable.
\end{lemma}
\begin{proof} $H$ is connected, simply-connected and has no non-trivial 
compact subgroup. Hence $H$ is the semi-direct product of a
simply-connected solvable group by a simply-connected semi-simple
Lie group $L$. This group $L$ is linear and has no non-trivial
compact subgroup. Hence $L$ is contractible and thus isomorphic to
$\widetilde{SO(2,1)}$, $\widetilde{SO(2,2)}$ or $\widetilde{SU(1,1)}$
($\,\,\widetilde{ }$~~denotes universal cover) or a product of these.
But such a group can not be linear as its centre is not finite. Hence
$L$ is trivial and $H$ is solvable.
\end{proof}
\begin{remark} In the situation $M=SU(1,n)/U(n)$, we have
\begin{eqnarray*}
  \su(1,n) &=& \{\left(
                  \begin{array}{cc}
                    il & \T\ol{b} \\
                    b & D \\
                  \end{array}
                \right)\suchthat l\in\R,b\in\C^n,D+\T\ol{D}=0,\,il+\Tr D=0\};
   \\
  P &=& \{\left(
            \begin{array}{cc}
              0 & \T\ol{b} \\
              b & 0 \\
            \end{array}
          \right)\suchthat b\in\C^n\};
   \\
  K &=& \{\left(
            \begin{array}{cc}
              -trD & 0 \\
              0 & D \\
            \end{array}
          \right)\suchthat D+\T\ol{D}=0\}.
\end{eqnarray*}
Up to conjugation, one can choose
\[
    \a=\R\left(
           \begin{array}{cc}
             0 & \T e_1 \\
             e_1 & 0 \\
           \end{array}
         \right)=:\R a 
\]
where $e_1$ is the first basis vector of $\C^n$. Then
\[
    \m=\left.\left\{\left(
           \begin{array}{ccc}
             -\frac{1}{2} tr E & 0 & 0 \\
             0 & -\frac{1}{2} tr E & 0 \\
             0 & 0 & E \\
           \end{array}
         \right)\right| E\in \u(n-1)\right\}.
\]
The homomorphism $\varphi$ is determined by the element $\varphi(a)$
of $\m$; up to conjugation, we may assume that $\varphi(a)$ belongs
to the standard maximal torus of $\u(n-1)$.
\end{remark}

%%%%%%%%%%%%%%%%
%%%%%%%%%%%%%%%%%%%

\section{Simply-transitive subgroups  of  the group of trans\-vections 
$G(M_A)$ for $M_A=\Sigma_A/\exp tA$, when $A^2=0$} \label{section:mueg0}

We have seen that for $A\ne 0$, $M_A=\Sigma_A/\exp tA$ is diffeomorphic to
$T(S^{q-1}\times\R^{p-q})\times\R^{2(n+1-p)}$ where $p$ and $q$ are
integers such that $1\leq p\leq n+1,1\leq q\leq p$. If $q=1$, the
manifold has two connected components, each diffeomorphic to
${\R}^{2n}$. A reasoning completely analogous to the one used in Theorem
\ref{thm 3.1} gives

\begin{proposition}
Let $M_A=\Sigma_A/\exp tA$, when $A^2=0$, be characterized
by the two integers $p,q$ where $1\leq p\leq n+1$ and $1\leq q\leq p$. If
$p=1$, $M_A$ is diffeomorphic to two copies of $\R^{2n}$; $M_A$ has the
flat symplectic connection and the translation group (which is the transvection group) acts 
simply-transitively on each component. If  $q\neq1,2,4$, $M_A$
does not admit a simply-transitive subgroup.
\end{proposition}

\noindent Of the remaining cases we look at the case where $p=2$ in detail.

\begin{theorem} 
Assume $A^2=0$ and $p=2$. Let $M_{Ao}$ be a connected component of $M_A=\Sigma_A/\exp tA$.
 Then $M_{Ao}$ admits a simply-transitive subgroup if and only
if $q=1$.
\end{theorem}

The proof of this theorem is split into two lemmas. We  first describe
all the subalgebras $\h$ of $\g=\p_1+[\p_1,\p_1]/\R A$ which are
supplementary to  $[\p_1,\p_1]/\R A$ in $\g$; the condition $q=1$ is necessary and sufficient
to have such algebras.. We then show that the connected subgroup of the
transvection group with algebra $\h$ acts simply-transitively on $M_{Ao}$.

%%%%%%%%%%%%
 We have shown in section \ref{section:two} that one can choose a basis 
 $\{e_1,e_2;f_a,a\leq2(n-1);e^*_1,e^*_2\}$ of $\R^{2n+2}$ in which
 $A=\left(
          \begin{array}{ccc}
            0 & 0 & I_2 \\
            0 & 0 & 0 \\
            0 & 0 & 0 \\
          \end{array}
        \right)$ and 
 \[
    \Omega=\left(
             \begin{array}{ccc}
               0 & 0 & \left(\begin{array}{cc}-1&0\\0&-\epsilon\end{array}\right) \\
               0 & \Omega^0 & 0 \\
            \left(\begin{array}{cc}1&0\\0&\epsilon\end{array}\right)& 0 & 0 
             \end{array}
           \right)
 \  \mathrm{with} \  \epsilon=\left\{\begin{array}{rl} 1& \mbox{ if } q=2\\
-1& \mbox{ if } q=1\end{array}\right.
\]
and a base point
$x_0=e^*_1\in\Sigma_A=\{(x,X,x_*)\,\vert\,(x_*^1)^2+\epsilon
(x_*^2)^2=1\}$ so that
\[
\p_1=\left\{
~\left(\begin{array}{ccc}
          \left( \begin{array}{cc}
              0 & -\epsilon p\\
              p & 0 \\
            \end{array}
          \right)
       &\left(\begin{array}{c}-\underline{P}\\0\end{array}\right)
      & \left( \begin{array}{cc}
              0 & \epsilon p'\\
              p' & 0 \\
            \end{array}
  \right)\\
             0 &0& \left(
            \begin{array}{cc}
             P &0
            \end{array}
          \right)
   \\
  0&0& \left( \begin{array}{cc}
              0 & -\epsilon p\\
              p & 0 \\
            \end{array}
          \right)
 \end{array}\right)
  \right\}
\]
with $p,p' \in \R,~ P\in\R^{2(n-1)}$ and $\underline{P}=\iota (P)\Omega^0$.

A subspace which is supplementary to  $[\p_1,\p_1]/\R A$ in  $\g=\p_1+[\p_1,\p_1]/\R A$ is of the form 
\[
\h_{B,\tilde{a},\tilde{b},\tilde{c},a,c}=\left\{
~K_{B,\tilde{a},\tilde{b},\tilde{c},a,c}(p,P,p')+\R A\,\vert \, 
p,p' \in \R, P \in \R^{2n-2} \right\}
\]
for a matrix $B\in \gl(\R^{2(n-1)})$, vectors
$\tilde{a},\tilde{b},\tilde{c} \in \R^{2(n-1)}$ and real numbers $a,c\in
\R$, with the matrix $K_{B,\tilde{a},\tilde{b},\tilde{c},a,c}(p,P,p')$
defined by
\[
\left(\begin{array}{ccc}
		\left( \begin{array}{cc}
				0 & -\epsilon p\\
				p & 0 \\
			\end{array}\right)
	 &\left(\begin{array}{c}-\underline{P}\\
	 -\epsilon(\underline{p\tilde{a}+BP+p'\tilde{c}})
	 \end{array}\right)
      & \left( \begin{array}{cc}
              -p'' & \epsilon p'\\
              p' & p''\\
\end{array}\right)\\
  ~&~&~\\
             0 &0& \left(
            \begin{array}{cc}
            ~&~\\
             P & ~~{p\tilde{a}+BP+p'\tilde{c}}\\
             ~&~\\
            \end{array}
          \right)
   \\
   ~&~&~\\
  0&0& \left( \begin{array}{cc}
              0 & -\epsilon p\\
              p & 0 \\
            \end{array}
          \right)
 \end{array}\right)
\]
where $p'':=ap+\Omega^0(\tilde{b},P)+cp'.$
Observe that the bracket is given by 
\[
[K_{B,\tilde{a},\tilde{b},\tilde{c},a,c}(p,P,p'),
K_{B,\tilde{a},\tilde{b},\tilde{c},a,c}(q,Q,q')]
=\left(\begin{array}{ccc} \left( \begin{array}{cc}
              0 & 0\\
              0 & 0 \\
            \end{array}
          \right)
       &\left(\begin{array}{c}-\underline{R}\\-\epsilon\underline{S}
       \end{array}\right)
      & \left( \begin{array}{cc}
             -r_1 & \epsilon r'\\
             r'& r_2\\
            \end{array}
  \right)\\
  ~&~&~\\
             0 &0& \left(
            \begin{array}{cc}
            ~&~\\
             R &  S\\
             ~&~\\
            \end{array}
          \right)
   \\
   ~&~&~\\
  0&0& \left( \begin{array}{cc}
              0 & 0\\
              0 & 0 \\
            \end{array}
          \right)
 \end{array}\right)
\]
with $\left\{ \begin{array}{l}R:=qBP-pBQ +(qp'-pq')\tilde{c},\\
S= \epsilon(-qP+pQ),\\
r'=-2p\Omega^0(\tilde{b},Q)+2q\Omega^0(\tilde{b},P) 
-2c(pq'-p'q)-\epsilon\Omega^0(p\tilde{a}+BP+p'\tilde{c},Q)\\
\qquad \qquad \qquad\mbox{}+\epsilon\Omega^0(q\tilde{a}+BQ+q'\tilde{c},P),\\
r_1= 2\epsilon(pq'-p'q)+2\Omega_0(P,Q), \\
r_2=  2\epsilon(pq'-p'q)-2\epsilon\Omega^0(p\tilde{a}
+BP+p'\tilde{c},q\tilde{a}+BQ+q'\tilde{c}).
\end{array}\right.$.\medskip

The subspace $\h_{B,\tilde{a},\tilde{b},\tilde{c},a,c}$ is a Lie
subalgebra of $\g$ if and only if $S=BR+r'\tilde{c} $ and
$\half(r_2+r_1)=\Omega^0(\tilde{b},R)+cr'$, i.e.{} if and only if for any
$p,p',q,q' \in \R$ and any $P,Q \in \R^{2n-2}$ one has
\begin{eqnarray}
\epsilon(-qP+pQ)&=&qB^2P-pB^2Q
 +(qp'-pq')B\tilde{c}+\left(-2p\Omega^0(\tilde{b},Q)
 +2q\Omega^0(\tilde{b},P)\right)\tilde{c}\label{*}\\
&&\mbox{}+\left( -2c(pq'-p'q)-\epsilon\Omega^0(p\tilde{a}
 +BP+p'\tilde{c},Q)+\epsilon\Omega^0(q\tilde{a}
 +BQ+q'\tilde{c},P)\right)\tilde{c}\nonumber\\
2\epsilon(pq'-p'q)&+&\Omega^0(P,Q)-\epsilon\Omega_0(p\tilde{a}
 +BP+p'\tilde{c},q\tilde{a}+BQ+q'\tilde{c})\nonumber\\
&=&\Omega^0(\tilde{b},qBP-pBQ +(qp'-pq')\tilde{c})
 +c\left(-2p\Omega^0(\tilde{b},Q)+2q\Omega^0(\tilde{b},P)\right)\label{**}\\
&&+c\left( -2c(pq'-p'q)-\epsilon\Omega^0(p\tilde{a}
 +BP+p'\tilde{c},Q)+\epsilon\Omega^0(q\tilde{a}+BQ+q'\tilde{c},P)\right)\nonumber. 
\end{eqnarray}
The terms in $q'P$ in equation (\ref{*}), $\epsilon
\Omega^0(q'\tilde{c},P)\tilde{c}$, imply that
\[
\tilde{c}=0
\]
and  equation (\ref{*}) is fulfilled provided we also have
\[
B^2=-\epsilon\Id.
\]
The terms in $pq-p'q'$ in equation (\ref{**}) then lead to
\begin{equation}\label{q1}
\epsilon=-1 \ (\mbox{ hence  } q=1),\ \ \mbox{ and }\  c^2=1.
\end{equation}
Terms in  $qP$ in equation (\ref{**}), $q\Omega^0(BP,\tilde{a})=\Omega^0(\tilde{b},
qBP)+c\left(2q\Omega^0(\tilde{b},P)-\Omega^0(q\tilde{a},P)\right)$ yield
\begin{equation}\label{tildeb}
\Omega^0(\tilde{b}, (B+2c)\,\cdot\,)=\Omega^0(\tilde{a},(c\Id-B)\,\cdot\,)
\end{equation}
and terms in $P,Q$ in equation (\ref{**}) yield
\begin{equation}\label{rel*}
\Omega^0((B-c\id)X,(B-c\Id)Y)=0 \qquad \forall X,Y \in \R^{2(n-1)}.
\end{equation}
Since $(B+2c\id)(B-2c\Id)=-3\Id$ and $(c\Id-B)(B-2c\Id)=3(cB-\Id)$,
$\tilde{b}$ is defined by equation (\ref{tildeb}) through
\begin{equation}\label{reltildeb}
\Omega^0(\tilde{b}, \,\cdot\,)=\Omega^0(\tilde{a},(\Id-cB)\,\cdot\,).
\end{equation}
Hence 

\begin{lemma} 
A symmetric space of dimension $2n$, $M_A=\Sigma_A/\exp tA$, $A^2=0$, with 
non-abelian solvable transvection group (i.e. p=2)  admitting a locally 
simply-transitive subgroup of the transvection group corresponds to the 
value $q=1$.
\end{lemma}

\begin{lemma} 
A symmetric space $M_A=\Sigma_A/\exp tA$, with $A^2=0$ and with a
non-abelian solvable transvection group and $q=1$ admits a family of
subgroups of the transvection group acting locally simply-transitively
on $M_A$. The algebra of such a subgroup is determined by an endomorphism
$B$ of $\R^{2(n-2)}$ satisfying
\begin{equation} \label{relB}
B^2=\Id \qquad \mbox{ and }\quad \Omega^0((B-c\id)X,(B-c\Id)Y)
 =0 \qquad \forall X,Y \in \R^{2(n-1)},
\end{equation} a vector $\tilde{a} \in\R^{2(n-2)}$, a real
number $a$ and a sign $c$, (i.e.~$c^2=1$);  it  is of the form
\[
\h_{B,\tilde{a},a,c}=\left\{
~K_{B,\tilde{a},a,c}(p,P,p')+\R A\,\vert \, p,p' \in \R, P \in \R^{2n-2} \right\}
\]
where 
\[
K_{B,\tilde{a},a,c}(p,P,p')=\left(\begin{array}{ccc}
          \left( \begin{array}{cc}
              0 &  p\\
              p & 0 \\
            \end{array}
          \right)
       &\left(\begin{array}{c}-\underline{P}\\(\underline{p\tilde{a}+BP})
       \end{array}\right)
      & \left( \begin{array}{cc}
              -p'' & -p'\\
              p' & p''
            \end{array}
  \right)\\
  ~&~&~\\
             0 &0& \left(
            \begin{array}{cc}
            ~&~\\
             P & ~~{p\tilde{a}+BP}\\
             ~&~\\
            \end{array}
          \right)
   \\
   ~&~&~\\
  0&0& \left( \begin{array}{cc}
              0 &  p\\
              p & 0 \\
            \end{array}
          \right)
 \end{array}\right)
\]
 with $p''=ap+\Omega^0(\tilde{a},(\Id-cB)P)+cp'$. 
\end{lemma}
Observe that 
\[
\Ad \left( \begin{array}{ccc}\Id&0&0\\ 0&S&0\\0&0&\Id
\end{array}\right)K_{B,\tilde{a},a,c}(p,P,p')=
 K_{SBS^{-1},S\tilde{a},a,c}\left(p,SP,p'\right)
\]
for any $S\in\Sp(\R^{2(n-1)},\Omega^0)$ so that, up to isomorphism given
by conjugation in the group of affine symplectomorphisms of $M_A$, $B$ can
be defined up to conjugation in $\Sp(\R^{2(n-1)},\Omega^0)$. Similarly
\[
\Ad \left( \begin{array}{ccc}\Id
&\left(\begin{array}{c}-\underline{u}\\0\end{array}\right)&0\\
0&\Id&\left(\begin{array}{cc}~&~\\ u&0\\ ~&~\end{array}\right) \\0&0&\Id
\end{array}\right)K_{B,\tilde{a},a,c}(p,P,p')=
 K_{B,\tilde{a}+u,a-c\Omega^0(u,\tilde{a}),c}\left(p,P,p'''\right)
\]
up to the addition of a multiple of $A$, with $p'''=p'+\Omega^0(u,p\tilde{a}+BP)$, so that one can assume, up to
conjugation in the group of affine symplectomorphisms of $M_A$, that
$\tilde{a}=0$, and
\[
\Ad \left( \begin{array}{ccc}\Id&0&\left(
\begin{array}{cc} r&0\\0&-r\end{array}\right)\\ 0&\Id&0\\0&0&\Id
\end{array}\right)K_{B,0,,a,c}(p,P,p')=
 K_{B,0,,a+2rc,c}\left(p,P,p'-2rp\right)
\]
so that one can assume  $a=0$.

We now proceed to prove that the connected subgroup of the transvection
group with algebra $\h_{B,\tilde{a},a,c}$ acts simply-transitively on a
connected component $M_{Ao}$ of $M_A$.

In view of the remarks above, it is enough to consider the case
$\h_{B,c}:=\h_{B,0,0,c}$ where $\tilde{a}=0$ and $a=0$.

As above, let $\Sigma_A=\{(x^1,x^2,X^1,\ldots,X^{2(n-1)},x_*^1,x_*^2)\,
\vert \,{x_{*}^1}^2-{x_{*}^2}^2=1\, \}$ and choose the connected
component $\Sigma_A^o=\{ (x^1,x^2,X^1,\ldots,X^{2(n-1)}, x_*^{1}=\ch
\alpha,x_*^{2}=\sh\alpha)\}$. Observe that $\partial_{\alpha}=\sh
\alpha\, \partial_{x_*^1}+\ch \alpha\, \partial_{x_*^2}$. Consider
$M_{Ao}=\Sigma_A^o/\exp tA$ and the canonical projection $\pi:\Sigma_A^o\to
M_{Ao}=\Sigma_A^o/\exp tA$.  We endow $M_{Ao}$ with coordinates
$(y^0,y^1,\ldots, y^{2(n-1)},\gamma)$ defined by
 \begin{eqnarray*}
 y^0\left( \pi(x^1,x^2,X^1,\ldots,X^{2(n-1)},\ch\alpha,\sh\alpha)\right )
   &=& -x^1\sh\alpha+x^2\ch\alpha\\
 y^a\left( \pi(x^1,x^2,X^1,\ldots,X^{2(n-1)},\ch\alpha,\sh\alpha)\right )
   &=&X^a, \quad 1\le a\le 2(n-1)\\
  \gamma\left( \pi(x^1,x^2,X^1,\ldots,X^{2(n-1)},\ch\alpha,\sh\alpha)\right )
    &=& \alpha.
\end{eqnarray*}
This shows in particular that $M_{Ao}$ is diffeomorphic to $\R^{2n}$.
Furthermore,
\[
\begin{array}{rcl@{\qquad}rcl}
\pi_* \partial_{x^1} &=& -\sh\alpha\, \partial_{y^0}&
\pi_* \partial_{x^2} &=& \ch\alpha \, \partial_{y^0}\\
\pi_* \partial_{X^a} &=& \partial_{y^a}& \pi_* \partial_\alpha &=&
\partial_{ \gamma} + (x^2 \sh \alpha
- x^1 \ch\alpha)\partial_{y^0}.\\
\end{array}
\]
If a bar denotes the horizontal lift on $\Sigma_A^o$ of vectors on
$M_{Ao}$, then one has
\[
\begin{array}{rcl}
\ol{\partial_{y^0}}
&=& \sh\alpha\,\partial_{x^1} + \ch\alpha\,\partial_{x^2}\\
\ol{\partial_{y^a}} &=& \partial_{X^a} + \sum_b
\Omega^0_{ab}X^b(\ch \alpha\,\partial_{x^1}
  + \sh\alpha\,\partial_{x^2})\\
\ol{\partial_\gamma} &=& \partial_\alpha +
\left(2x^1\sh\alpha\,\ch\alpha-x^2(\ch^2\alpha
+ \sh^2\alpha)\right)\partial_{x^1}\\
&&\qquad  \mbox{}+ \left(x^1(\ch^2\alpha + \sh^2\alpha)
-2 x^2\sh\alpha\,\ch\alpha\right)\partial_{x^2}\\
\end{array}
\]
Hence the symplectic form on $M_{Ao}$ has the form
\[
\omega = dy^0\wedge d\gamma+ \frac12 \sum_{1\le a,b \le 2(n-1)} \Omega^0_{ab}\,
dy^a \wedge dy^b
\]
showing in particular that these are global Darboux coordinates.

Since the projection $\pi:\Sigma_A\to\Sigma^o_A/\exp tA=M_{Ao}$ is
equivariant with respect to the group of linear symplectic transformations
of $\R^{2(n+2)}$ commuting with $A$, the fundamental vector fields on
$M_{Ao}$ associated to the elements of $\h_{B,c}$ are the projections of
the corresponding vector fields on $\Sigma_A^o$.

The fundamental vector field on $\Sigma_A^o$ associated to 
\[
K_{B,c}(p,P,p')=\left(\begin{array}{ccc}
          \left( \begin{array}{cc}
              0 &  p\\
              p & 0 \\
            \end{array}
          \right)
       &\left(\begin{array}{c}-\underline{P}\\(\underline{BP})
       \end{array}\right)
      & \left( \begin{array}{cc}
              -cp' & -p'\\
              p' & cp'
            \end{array}
  \right)\\
  ~&~&~\\
             0 &0& \left(
            \begin{array}{cc}
            ~&~\\
             P & ~BP\\
             ~&~\\
            \end{array}
          \right)
   \\
   ~&~&~\\
  0&0& \left( \begin{array}{cc}
              0 &  p\\
              p & 0 \\
            \end{array}
          \right)
 \end{array}\right)
\]
is given by 
\begin{eqnarray*}
\left(K_{B,c}(p,P,p')\right)^{*\Sigma_A^o}_{(x,X,\alpha)}
&=&\left( -px^2+\Omega^0(P,X)+p'(c\ch \alpha+\sh\alpha)\right) \partial_{x^1}\\
&& +\left( -px^1-\Omega^0(BP,X)-p'(\ch\alpha+c\sh\alpha) \right) \partial_{x^2}\\
&& -\sum _{a=1}^{2(n-1)}(\ch\alpha\, P^a + \sh\alpha (BP)^a)\partial_{X^a}\\
&&-p\,\partial_\alpha
\end{eqnarray*}
and the fundamental vector field on $M_{Ao}$ is
$\left(K_{B,c}(p,P,p')\right)^{*M_{Ao}}_{\pi(x,X,\alpha)}=\pi_*\left(K_{B,c}
(p,P,p')\right)^{*\Sigma_A^o}_{(x,X,\alpha)}$  so that
\begin{eqnarray*}
\left(K_{B,c}(p,P,p')\right)^{*M_{Ao}}_{(y^0,Y=(y^1,\ldots,y^{2(n-1)}),\gamma)}
&=& -p\partial_{\gamma}-\sum_{a=1}^{2(n-1)} \left(\ch\gamma\, P^a+\sh\gamma\, 
(BP)^a\right)\partial_{y^a}\\
&&\kern-2cm\mbox{}-\left( \Omega^0(P,Y)\sh\gamma+\Omega^0(BP,Y)\ch\gamma+p'
(\sh\gamma +c\ch\gamma)^2\right)\partial_{y^0}.
\end{eqnarray*}

Recall that $B^2=\Id$  so that the matrix $(\ch\gamma \Id+\sh\gamma\,
B)$ is invertible for all $\gamma$. Hence the fundamental vector fields
are linearly independent at each point of $M_{Ao}$. Thus any orbit of the
connected subgroup with algebra $\h_{B,c}$ is open. The connectedness of
$M_{Ao}$ implies that the action is transitive. Hence

\begin{theorem} 
Each connected subgroup $H_{B,\tilde{a},a,c}$ of the transvection group
of $M_{Ao}$, with algebra $\h_{B,\tilde{a},a,c}$ acts globally
simply-transitively on $M_{Ao}$. In particular, the groups
$H_{B,\tilde{a},a,c}$ are symplectic groups and symmetric spaces.
\end{theorem}

To conclude this paragraph, we investigate the question of the strongly
Hamiltonian character of the action of the group $H_{B,c}$ on $M_{Ao}$.

\begin{proposition} The action of the group $H_{B,c}$ on $M_{Ao}$ is 
strongly Hamiltonian if and only if $B=c\Id$. The group
$H_{c\Id,c}$ is the $1$-dimensional extension of the Heisenberg group
of dimension $2n-1$ by the dilation automorphism.
\end{proposition}
\begin{proof} 
The contraction of the  $2$-form $\omega = dy^0\wedge d\gamma+ \frac12
\sum_{1\le a,b \le 2(n-1)} \Omega^0_{ab}\, dy^a \wedge dy^b$ on $M_{Ao}$ by
the fundamental vector field $\left( K_{B,c}(p,P,p')\right)^{*M_{Ao}}$ is
the differential of a function $f_{B,c}^{(p,P,p')}$
\begin{eqnarray*}
\iota\left(\left( K_{B,c}(p,P,p')\right)^{*M_{Ao}}\right)\omega
&=&pd y^0-p'e^{2c\gamma}d\gamma\\
&&-\left(\Omega^0(P,Y)\sh \gamma+\Omega^0(BP,Y)\ch \gamma \right) d \gamma\\
&&-\sum_{a,b=1}^{2(n-1)}\left( P^a \ch \gamma + (BP)^a 
\sh \gamma \right)\Omega^0_{ab} dy^b\\
&=& d f_{B,c}^{(p,P,p')}
\end{eqnarray*}
with
\[
f_{B,c}^{(p,P,p')}(y^0,Y,\gamma)
:=p y^0-\frac{1}{2c}p'e^{2c\gamma}-\Omega^0(P,Y)\ch \gamma-\Omega^0(BP,Y)\sh \gamma.
\]
The Lie algebra structure on $\h_{B,c}$ is given by
\[
\left[K_{B,c}(p,P,p'),K_{B,c}(q,Q,q')\right]
=K_{B,c}(0,qBP-pBQ,-2c(pq'-qp')+\Omega^0(BP,Q)+\Omega^0(P,BQ))
\]
  and 
\begin{eqnarray*}
\left(K_{B,c}(p,P,p')\right)^{*M_{Ao}}\left( f_{B,c}^{(q,Q,q')}\right)
&=&f_{B,c}^{(0,qBP-pBQ,-2c(pq'-qp')+\Omega^0(BP,Q)+\Omega^0(P,BQ))}\\
&&\qquad +\frac{1}{2}\left( \Omega^0(BP,BQ)-\Omega^0(P,Q)\right)
\end{eqnarray*}
so the action is strongly Hamiltonian if and only if
\begin{equation}\label{Bsympl}
\Omega^0(BP,BQ)=\Omega^0(P,Q) \qquad \forall P,Q;
\end{equation}
since $B^2=\Id$, equation(\ref{Bsympl}) also implies
$\Omega^0(BP,Q)=\Omega^0(P,BQ)$; the relation (\ref{rel*})
$\Omega^0((B-c\Id)P,(B-c\Id)Q)=0$, becomes
$2\Omega^0(P,Q)-2c\Omega^0(BP,Q)=0$ and this is equivalent to
\[
B=c\Id.
\]
Note that the Lie algebra structure of $\h_{c\Id,c}$ with $c^2=1$ is
\[
\left[K_{c\Id,c}(p,P,p'),K_{c\Id,c}(q,Q,q')\right]
=K_{c\Id,c}\left(0,c\left(qP-pQ\right),2c\left(qp'-q'p+\Omega^0(P,Q)\right)\right).
\]
Hence the derived algebra $\h'_{c\Id,c}=[\h_{c\Id,c},\h_{c\Id,c}]$ is
isomorphic to the Heisenberg algebra $\h_{2n-1}$ in dimension $2n-1$
\[
\h'_{c\Id,c}\simeq\left\{ (P,p')\,\vert\, P\in \R^{2(n-1)},p'\in\R\right\} 
\ \mathrm{with}\  [(P,p'),(Q,q')]=\left(0,2c\Omega^0(P,Q)\right)\simeq\h_{2n-1},
\]
and 
\[
\h_{c\Id,c}=\h'_{c\Id,c}\oplus \R D \ \ \mathrm{with}\ \ D \left( (P,p')\right)
:=(-cP,-2cp')\simeq \mathcal{K}_{2n}.
\]
\end{proof}
%***********************************************************************%

%%%
%%%%%
%%%%% To add a Table of Contents line for references uncomment next line
%%%%%
%%%%\addcontentsline{toc}{section}{\protect\numberline{}References}
%%%
\providecommand{\bysame}{\leavevmode\hbox
to3em{\hrulefill}\thinspace}
\providecommand{\MR}{\relax\ifhmode\unskip\space\fi Maths Reviews: }
% \MRhref is called by the amsart/book/proc definition of \MR.
\providecommand{\MRhref}[2]{%
  \href{http://www.ams.org/mathscinet-getitem?mr=#1}{#2}
} \providecommand{\href}[2]{#2}

\end{document}